\documentclass[a4paper]{amsart}
\usepackage[english]{babel}
\usepackage{amssymb}
\usepackage{amsmath}
\usepackage{amsthm}
\usepackage{graphicx}
\usepackage{fancyhdr}
\usepackage{bbm}
\usepackage{multirow}
\newfont{\bssten}{cmssbx10}
\newfont{\bssnine}{cmssbx10 scaled 900}
\newfont{\bssdoz}{cmssbx10 scaled 1200}

\usepackage{latexsym,amssymb,amsthm,amsxtra}
\usepackage{amsmath}
\usepackage{stmaryrd}
\usepackage{mathrsfs}
\usepackage{tikz}
\usepackage{pgf}
{\bf}{\it}
{\bf}{\it}
\newtheorem{theorem}{Theorem}
\newtheorem{definition}{Definition}
\newtheorem{lemma}{Lemma}
\newtheorem{remark}{Remark}
\newtheorem{proposition}{Proposition}
\newtheorem{corollary}{Corollary}

\renewcommand{\leq}{\leqslant}
\renewcommand{\geq}{\geqslant}
\renewcommand{\le}{\leq}				% Synonyme
\renewcommand{\ge}{\geq}				% Synonyme
\newcommand{\eq}{\Longleftrightarrow}
\newcommand{\PP}{\ensuremath{\mathbb{P}}}

\newcommand{\x}{\mathbf x}
\newcommand{\bz}{\mathbf z}
\newcommand{\bu}{\mathbf z}
\newcommand{\by}{\mathbf y}
\newcommand{\ba}{\mathbf a}
\newcommand{\bb}{\mathbf b}
\newcommand{\bx}{\overline{\x}}
\newcommand{\bm}{\overline{\m}}
\newcommand{\tx}{\widetilde{\x}}
\newcommand{\tm}{\widetilde{\m}}
\newcommand{\olb}{\ol{\bb}}
\newcommand{\ollb}{\oll{\bb}}
\newcommand{\cc}{\mathbf c}
\newcommand{\olc}{\ol{\cc}}
\newcommand{\ollc}{\oll{\cc}}

\newcommand{\limt}{\lim\limits}

\newlength{\boxrulewidth}
\newlength{\boxrulesep}
\newlength{\gauchelong}
\newlength{\gauchelongadd}
\newlength{\droitelong}
\newlength{\colonnealong}
\newlength{\colonneblong}
\newlength{\colonneclong}

\newlength{\maxhaut}
\newlength{\maxbas}

\newcommand{\writeifexist}[1]{%
        \ifthenelse{\equal{#1}{\null}}{\null}{#1&}%
}

\newcommand{\lengthifexist}[2]{%
       \ifthenelse{\not\equal{#1}{\null}}{%
                \settowidth{#2}{$#1$}
                }{\setlength{#2}{0cm}}%
}

\makeatletter

\def\arrayboxold#1{
    \@ifnextchar[%]
    	{\iarraybox{#1}}%
    	{\ivarraybox{#1}{\null}{\null}[\null]}}

\def\iarraybox#1[#2]{%
    \@ifnextchar[%]
    	{\iiarraybox{#1}{#2}}%
    	{\ivarraybox{#1}{#2}{\null}[\null]}}

\def\iiarraybox#1#2[#3]{%
    \@ifnextchar[%]
  	  {\iiiarraybox{#1}{#2}{#3}}%
    	{\ivarraybox{#1}{#2}{#3}[\null]}}

\def\iiiarraybox#1#2#3[#4]{%
    \@ifnextchar[%]
  	  {\ivarraybox{#1}{#2}{#3}[#4]\relax}%
    	{\ivarraybox{#1}{#2}{#3}[#4]}}%

\def\ivarraybox#1#2#3[#4]#5{%
        \settowidth{\gauchelong}{$#1\ $}%
        \settoheight{\maxhaut}{$#1#2#3#4#5$}
        \settodepth{\maxbas}{$#1#2#3#4#5$}
        \lengthifexist{#2}{\colonnealong}%
        \lengthifexist{#3}{\colonneblong}%
        \lengthifexist{#4}{\colonneclong}%
        \raisebox{-\maxbas-\boxrulewidth-\boxrulesep}{%
                \rule{\boxrulewidth}{%
                \maxbas+\maxhaut+2\boxrulewidth+2\boxrulesep}}%
        \raisebox{\maxhaut+\boxrulesep}{%
                \makebox[0cm][l]{\rule{%
                \gauchelong+\gauchelongadd+\colonnealong+\colonneblong%
				+\colonneclong}{%
                \boxrulewidth}}}%
        \raisebox{-\maxbas-\boxrulewidth-\boxrulesep}{%
                \makebox[0cm][l]{\rule{%
                \gauchelong+\gauchelongadd+\colonnealong+\colonneblong%
				+\colonneclong}{%
                \boxrulewidth}}}%
        \hskip\boxrulesep%
        #1&%
   		\writeifexist{#2}%
		\writeifexist{#3}%
        \writeifexist{#4}%
        #5%
        \settowidth{\droitelong}{$#5\ \eq $}
        \settoheight{\maxhaut}{$#1#2#3#4#5$}
        \settodepth{\maxbas}{$#1#2#3#4#5$}
        \hskip\boxrulesep%
        \hskip-\droitelong%
        \raisebox{\maxhaut+\boxrulesep}{%
                \makebox[0cm][l]{\rule{\droitelong}{\boxrulewidth}}}
        \raisebox{-\maxbas-\boxrulesep-\boxrulewidth}{%
                \makebox[0cm][l]{\rule{\droitelong}{\boxrulewidth}}}
        \hskip\droitelong%
        \raisebox{-\maxbas-\boxrulesep-\boxrulewidth}{%
                \rule{\boxrulewidth}{%
                \maxbas+\maxhaut+2\boxrulesep+2\boxrulewidth}}
}

\makeatother
\renewcommand{\min}{\mathop{\mathrm{min}\,}\limits}
\renewcommand{\max}{\mathop{\mathrm{max}\,}\limits}
\renewcommand{\inf}{\mathop{\mathrm{inf}\,}\limits}
\renewcommand{\sup}{\mathop{\mathrm{sup}\,}\limits}

\newfont{\cmmd}{rsfs10 at 10pt}
\newfont{\pcmmd}{rsfs10 at 7pt}

\newcommand{\et}{^{\star}}

\newcommand{\acco}[1]{\left\{#1\right\}} % donne {.......}
\newcommand{\pare}[1]{\left(#1\right)}
\newcommand{\croc}[1]{\left[#1\right]}
\newcommand{\crocc}[1]{[\![#1]\!]}

\newcommand{\virg}{\raise1pt\hbox{,}\ }
\newcommand{\pnt}{\raise1pt\hbox{.}\ }

\newcommand{\dsum}{\displaystyle\sum}

\def\dsum{\displaystyle\sum}

 \def\m{\mathbf m}
 
 \def\R{\mathbb R}
 \def\N{\mathbb N}
 \def\E{\mathbb E}
 \def\Z{\mathbb Z}
 \def\pasc#1{\textcolor{black}{#1}}
\def\pasca#1{\textcolor{black}{#1}}

 \newcommand\suite[1]{\left\{#1_n;\,n\in\N\right\}}
 \newcommand{\gre}{\mathbf e}
\newcommand\esp[1]{{\mathbb E}\left[#1\right]}
\newcommand\pr[1]{{\mathbb P}\left(#1\right)}
\newcommand{\devers}[2]{\colon #1 \rightarrow #2}
\newcommand{\fonccc}[4]{\left\{\begin{array}{ccc}#1 & \longrightarrow & #2 \\ #3& \longmapsto & #4\end{array}\right.}

\newcommand{\Ilun}{\mathcal{I}_{1}}
\newcommand{\Ilq}{\mathcal{I}_{3}}

\newcommand{\Cld}{\mathcal{C}_{2}}

\newcommand{\Ip}{\mathcal{I}'}

\newcommand{\Cp}{\mathcal{C}'}
\newcommand{\oll}[1]{\overline{\overline{#1}}}

\newcommand{\ol}[1]{\overline{#1}}

\newcommand{\BB}{\mathcal{B}}

\makeatother
\title[Dynamic programming for the GM model on the N-graph]{Dynamic programming for the stochastic matching model on general graphs: the case of the `N-graph'}
\author[L. Jean and P. Moyal]{Loïc Jean and Pascal Moyal\\\\
\scriptsize{IECL,Université de Lorraine and Inria PASTA}}
\begin{document}
\maketitle

\begin{abstract}
    In this paper, we address the optimal control of stochastic matching models on general graphs and single arrivals having fixed arrival rates, 
    as introduced in \cite{MaiMoy16}. 
    On the `N-shaped' graph, by following the dynamic programming approach of \cite{BCD19}, we show that a `Threshold'-type policy on the diagonal edge, with priority to the extreme edges, is optimal for the discounted cost problem and linear holding costs. 
    \end{abstract}

\section{Introduction}
In this work, we consider a general stochastic matching (GM) model, as defined in \cite{MaiMoy16}. It is roughly defined as follows: items of different classes enter a system one by one, and are of different classes. We let $V$ be the set of classes, and also define a compatibility graph $G=(V,E)$ on the set of 
classes. If two nodes (i.e., classes) $i$ and $j$ in $V$ share an edge in $G$, we consider that any couple of items of respective classes $i$ and $j$ are {\em compatible}. Then, upon arrival into the system, any element can either be stored in a buffer, or matched with a compatible item that is already present in the buffer. 
It is then the role of the matching policy to determine the match of a given item, in case of multiple choices. We say that the matching policy is {\em greedy} if 
any incoming item that finds compatible items in line is necessarily matched upon arrival. Otherwise, in the {\em non-greedy} case, one might want to keep two compatible items together in line, to wait for a more profitable future match. 

The GM model is a variant of the the so-called Bipartite stochastic matching (BM) model introduced 
in \cite{CKW09} for the {\em Fist Come, First matched} (FCFM) policy, and generalized in \cite{BGM13}. In this class of models that typically represent skill-based queueing systems, the compatibility graph $G$ is bipartite (there are classes of {\em customers} and classes of {\em servers}) and arrivals occur by pairs customer/server - an assumption that might appear as unpractical, in many applications.  The `general' stochastic matching model (GM) we consider in this work, stands for a system in which the compatibility graph is general (i.e., non necessarily bipartite), and arrivals occur by single items rather than couples, entailing a widely different dynamics, and a different analysis. For the GM model, an important line of literature first considered the stability problem. Necessary conditions for stability were shown in \cite{MaiMoy16}, entailing in particular that no GM model on a bipartite graph can be stabilizable. 
Then, \cite{MoyPer17,MBM21,JMRS23} obtained various stability results for greedy policies and non-bipartite graphs: 
\cite{MoyPer17} shows that greedy class-uniform policies never maximize the stability region, and that neither do greedy strict priority policies in general; 
By applying the dynamic reversibility argument in \cite{ABMW18} to general graphs and single arrivals, \cite{MBM21} proves that the greedy  {\em First Come, First Matched} policy maximizes the stability region, and construct explicitly the invariant measure and the matching rates between classes (a result that is completed in \cite{comte22} by an insightful comparison to order-independent queues). Then, \cite{JMRS23} proved that greedy matching policies of the `Max-Weight' type also maximize the stability region.% even in the case where some items renege from the system before getting matched. 

The performance optimization of this class of systems have recently been considered from the point of view of access control, by providing explicit procedures to {\em construct} the arrival rates in order to achieve given matching rates, for a {\em fixed} matching policy: see \cite{CMB21,BMM23}. Another approach has recently received an increasing interest: designing (greedy or non-greedy) matching policies able to maximize a given reward or to minimize a given cost, in the long run. In \cite{NS19}, for a vast class of matching structures, a variant of the greedy primal-dual algorithm was shown to be optimal for the long-term average matching reward, where the rewards are put on the matchings. Regarding the minimization of cumulative holding costs, \cite{GW15} derives a lower bound for the long-run cumulative holding cost, and shows that this bound can be asymptotically approached under some conditions. 
Then, \cite{BM15} constructs a policy that is approximately optimal with bounded regret, in the heavy-traffic regime. 
It is also shown in \cite{KAG23} that greedy policies are hindsight optimal (i.e., nearly maximize the total value function simultaneously at all times), and that a {\em static} hindsight optimal greedy policy can be prescribed, whenever $G$ is a tree. For fixed arrival rates, \cite{BCD19} show that, under the stability conditions prescribed in \cite{BGM13}, a matching policy of the `threshold' type is optimal for holding costs, for a BM model, whenever $G$ is the `N'-graph of Figure \ref{Fig:Ngraph} below, for the discounted cost as well as the average cost problem. 

In this work, we show that the result of \cite{BCD19} concerning a BM model on the `N'-graph, also holds for the GM model. Specifically, by using the tools of dynamic programming, we show that a `threshold'-type policy on the diagonal edge of $G$ is optimal, for the discounted cost problem. The analysis follows the main steps of the main proof in \cite{BCD19}; However, the technical arguments change significantly, due to the fact that arrivals are single rather than pairwise. In particular, it is remarkable that, due to Theorem 2 in \cite{MaiMoy16}, the system is not stabilizable by any 
matching policy (greedy or not). However, optimality is intended for the discounted cost problem (and for linear holding costs), guaranteeing the well-posedness of the optimization scheme, despite the instability of the process at hand. By doing so, for this simple graph, we show what appears as the first optimal control result for general stochastic models having fixed arrival rates. 

This paper is organized as follows. In Section \ref{sec:model} we start with some preliminary, and the definition of the general stochastic matching model at hand. Then, in Section \ref{sec:N} we introduce our dynamic programming problem and our main result, stating the optimality of a threshold-type policy on the diagonal edge with priority on the other edges.  
The proof of our main result is left to Section \ref{sec:proof}.

\section{Preliminary}\label{sec:model}

\subsection{Notation}
Let $\R$, $\N$ and $\N^*$ denote the sets of real numbers, non-negative and positive integers, respectively. 
For any $d\in \N^*$, we let $\llbracket 0,d-1 \rrbracket$ be the integer interval $\{0,1,2,\cdots,d-1\}$. 
For any $i\in\llbracket 0,d-1 \rrbracket$, we let $\gre_i$ be the $i$-th vector of the canonical basis of $\R^d$. 
Vectors of $\N^d$ will in general be denoted in bold, that is, we write e.g. $$\x=(x_0,\cdots,x_{d-1}).$$ 

\noindent For an un-directed graph $G=(V,E)$, we write $i - j$ if $\{i,j\}\in E$. For any $i\in V$, we let 
\[E(i)=\left\{j\in E\,:\,i-j\right\}.\]

\noindent Throughout, all random variables are defined on a common probability space $(\Omega,\mathscr F,\mathbb P)$. 

\subsection{A general stochastic matching model}

We consider a stochastic matching model, as defined in \cite{MaiMoy16}: 
Consider the `N-graph', represented in Figure \ref{Fig:Ngraph}. It is denoted by 
$$G=(V,E):=\left(\left\{0,1,2,3\right\},\left\{\ell_1,\ell_2,\ell_3\right\}\right)
:=\left(\left\{0,1,2,3\right\},\left\{\{0,1\},\{1,2\},\{2,3\}\right\}\right).$$

\begin{figure}[h!]
\begin{center}
\begin{tikzpicture}
\draw (0,0) -- (0,2);
\draw (0,2) -- (2,0);
\draw (2,0) -- (2,2);
\fill (0,0) circle (2pt) node[below] {\small{0}} ;
\fill (0,1) node[left] {\small{$\ell_1$}} ;
\fill (0,2) circle (2pt) node[above] {\small{1}} ;
\fill (1,1) node[right] {\small{$\ell_2$}} ;
\fill (2,0) circle (2pt) node[below] {\small{2}} ;
\fill (2,1) node[right] {\small{$\ell_3$}} ;
\fill (2,2) circle (2pt) node[above] {\small{3}} ;
\end{tikzpicture}
\caption[smallcaption]{The `N-graph'.}
\label{Fig:Ngraph}
\end{center}
\end{figure}
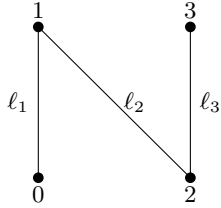

The dynamics of the system can be described as follows: Items enter one by one, in discrete time, and are of different {\em classes}. Upon arrival, the class $i\in V$ of the incoming item is drawn, independently of everything else, from a fixed probability distribution $\mu$ on $V$, and we say in that case that the item is a {\em $i$-item}. In other words we see the arrival process as an infinite $\N^4$-valued IID sequence 
$\suite{A}$ such that for all $n\in\N$ and $i\in V:=\llbracket 0,3 \rrbracket$,   
\[\pr{\mbox{The incoming item at step $n$ is a $i$-item}}:=\pr{A_n=\gre_i}=\mu(i),\]
where $\mu$ is a probability measure over $V$. For a node $j\in E(i)$, any $j$-item is said to be {\em compatible} with the $i$-item. 

Upon arrival, any $i$-item faces the following situation: 
\begin{enumerate}
    \item Either it does not find a compatible item in line, in which case it is stored in line;
    \item Or, there exists at least one compatible item in line at that time, in which case it is the role of the {\em decision rule}, to be defined hereafter, to determine a {\em match} for the $i$-item. Following the latter policy, we have the following sub-cases:
    \begin{enumerate}
        \item[2a)] Either the $i$-item is matched right away with a $j$-item for a given $j\in E(i)$, in which case both items leave the system right away. 
        \item[2b)] Or, the $i$-item is put in line to wait for a potential better match, despite the presence of one (or more) compatible item(s) in line. 
    \end{enumerate}
\end{enumerate}
At any time $n\in\N$, we denote by $X_n \in \N^4$, the state of the buffer just after time $n$. 
Namely, for any $i\in \llbracket 0,3 \rrbracket$ we denote by $X_n(i)$, the number of stored $i$-items in line just after time $n$, and set 
\[X_n=\left(X_n(0),\cdots,X_n(3)\right).\] 

\medskip

The decision rule is to be {\em Markovian}, if the choice of match at any time $n$ is prescribed by an operator $u_n\devers{\N^4}{\N^4}$ that is deterministic, or independent of the filtration generated by the $X_p,\,p\le n$, and such that 
\[X_{n+1}=u_n(X_n)+A_{n+1}.\]
The decision rule is said to be {\em greedy} if sub-case 2b) above is not allowed, namely, whenever an incoming item finds compatible items in line, it must be matched with one of those right away. An {\em admissible policy} $\pi$ is a sequence $\suite{u}$ of Markovian decision rules. We then say that the policy $\pi$ is {\em structured} by the decision rules $u_n,\,n\in\N$. 
%A policy $\pi$ is said to be {\em admissible}, if it structured by decision rules that are deterministic, or independent of everything else and of one another. 
We denote by \textsc{adm}, the set of admissible policies. Notice that a decision rule may be constant, namely $u_n=u$ for all $n$, in which case we say that then the policy $\pi$ is {\em stationary}. 
The state space of $\suite{X}$ is a subset $\mathbb X_\pi$ of $\N^4$ that depends on the policy $\pi$. For instance, if $\pi$ is greedy, then 
\[\mathbb X_\pi=\left\{\x=(x_0,\cdots,x_{3})\in \N^4\,:\,x_ix_j=0\mbox{ for any }\{i,j\}\in E\right\}.\]
Under a stationary policy $\pi$, 
the process $\suite{X}$ is thus a $\mathbb X_\pi$-valued homogeneous Markov chain, and for all state $\x\in \mathbb X_\pi$ 
we have for all $v\in V$, 
$$\pr{u(\x)+\gre_v\mid \x}:=\mathbb{P}\left(X_{n+1}=u(\x)+\gre_v \,|\, X_{n}=\x\right)=\mu(v).$$
%The following observation follows from \cite{SJM23}, 
\begin{definition}
	Any state $\x=(x_0,\cdots,x_3)\in\mathbb X_\pi$ corresponds uniquely to a block model graph 
        $\mathbf G_\x=(\mathbf V_\x,\mathbf E_\x)$ having $|\x|$ nodes, in which there are 
        $x_i$ nodes of type $i$ for all $i\in\llbracket 0,3 \rrbracket$ and, for any $i\ne j \in \llbracket 0,3 \rrbracket$, and any pair of nodes $u$ and $v$ of respective types $i$ and $j$, there exists an edge between $u$ and $v$, that is, $\{u,v\}\in \mathbf E_\x$, if and only if $\{i,j\}\in E$. Note that $\mathbf G_\x$ has no edge if $\pi$ is greedy. 
	\end{definition}
 
	\begin{definition}  
	For any $\x\in\mathbb X_\pi$, we denote by $\mathbf M_\x$, the set of all possible matches from state $\x$, that is, the set of all matchings 
        on $\mathbf G_\x$ or in other words, of all sub-graphs of $\mathbf G_\x$ such that each node has degree 0 or 1. 
	\end{definition}
For any pair of states $\x,\by\in \mathbb X_\pi$, we write $\x \to \by$ whenever $\pr{\by \mid \x}>0$. 
In the sequel, we denote by $\mathbb{E}^\pi_\x$ the expected value of r.v.'s, under the policy $\pi$, given that the initial state is 
$\x\in \mathbb X_\pi$. 
In the discussion that follows, we identify: 
\begin{itemize}
    \item Any edge $e=\{i,j\}\in E$ with the vector $\gre_i+\gre_j\in \N^4$;
    \item For any $\x\in\N^4$, any matching $\mathbf m\in \mathbf M_\x$ with the vector $\mathbf z\in \N^4$ such that 
    for all $i\in V$, $z_i$ is the number of items of class $i$ matched in $\mathbf m$. In other words, 
    \[\m\equiv \sum_{\stackrel{v\,\in \,\m:}{v\mbox{\scriptsize{ is of type }}i}}\gre_i,\]
    in a way that $\x-\mathbf m$ indeed represents the state of a system after executing all matchings in $\mathbf m $ in a system in state $\x$. 
\end{itemize}

 %For any $\x$ and any matching 
%$\mathbf m\in \mathbf M_\x$, we identify $\m$ with the set of nodes that appear in $\m$, namely, 

%In particular, the vector $\x-\mathbf m\in\mathbb X_\pi$ can then be interpreted as the state reached after executing all matches in $\m$, 
%in a system represented by state $\x$. 

\section{Optimization of the `N-graph' model }
\label{sec:N}

\subsection{Dynamic programming}
\noindent We now introduce our optimization problem. 

\begin{definition} %[Linear cost function]
 \label{defi:hypothese} \  
\noindent An {\em admissible linear cost function} is a linear mapping $c$ of the form 
$$c:\fonccc{\N^4}{\R_+}{\mathbf x:=(x_0,x_1,x_2,x_3)}{c_0x_0+c_1x_1+c_2x_2+c_3x_3}\,,$$
\noindent whose coefficients $c_0,c_1,c_2,c_3$ are non-negative, and such that $$c_2\leq c_0 \ \text{ and } \ c_1\leq c_3.$$
\end{definition}

Let $c:\N^4\to \R_+$ be an admissible linear cost function. For a fixed discount factor $\gamma\in (0,1)$ and a fixed initial state $\x\in \mathbb X_\pi$, our aim is to minimize over $\pi$, the so-called {\em discounted cost}

%	\noindent We will denote by 
	
%	\noindent Il existe plusieurs problèmes d'optimisation : 
	
%	\begin{itemize}
%	\item Le problème du coût moyen à horizon infini, où l'on cherche à minimiser\footnote{ou bien la $\liminf$ et la $\limsup$ de cette suite lorsque sa limite n'existe pas.} $$g^\pi(x)=\limt_{N\to+\infty}\dfrac{1}{N}\dsum_{n=0}^{N}\mathbb{E}^\pi_x\croc{c(X_n)}$$
	
%	\item Le problème du coût réduit à horizon infini: on se préoccupe moins du futur, et la décroissance exponentielle de notre intérêt varie avec le paramètre $\gamma$ : on cherche à minimiser 
\begin{equation}
\label{eq:defdiscount}
v^\pi_\gamma(\x)=\limt_{N\to+\infty}\dsum_{n=0}^{N}\gamma^n\mathbb{E}^\pi_\x\croc{c(X_n)}.
\end{equation}
%$v$ is called the {\em value function} of our optimization problem. For any feasible trajectory starting from $\x$, that is, any sequence $\tilde{\x}:=\left\{\tilde{\x}_n,\,n\in\N\right\}$ that belongs to $(X^\pi)^\N$ for some admissible policy $\pi$, and such that 
%$X_0=\tilde{\x}_0$ and, for any $n\in\N$, \[\pr{\tilde{\x}_{n+1} \mid \tilde{\x}_n}>0,\]
%we denote 
%\[u(\tilde{\x})=\dsum_{n=0}^{+\infty}\gamma^n c(\tilde{\x}_n),\]
%the total discounted cost of the trajectory $\tilde{\x}$. Then, we retrieve that 
%$$v(\x)=\inf_{\tilde{\x}\mbox{ feasible }}u(\tilde{\x}).$$

%	\begin{rem} Pour une métrique qui rend $\mr{Adm}(x)$ compact, la fonction $u$ est continue. Donc la borne inférieure de la fonction de valeur est en fait un minimum. La fonction de valeur correspond donc au coût optimal des décisions futures envisageables. 
	
%	\end{rem}
	
%\item Le problème du coût moyen à horizon fini, où l'on cherche à minimiser $$m^\pi_N(x)=\mathbb{E}\croc{\dsum_{n=0}^{N}c(X_n)}$$
%\end{itemize}
	
%Dans la suite, on s'intéressera plutôt aux deux premiers problèmes de cette liste. 

\begin{definition}%[$L^\gamma$ Operator]  
Let $v$ be a real function on $\N^4$, and $A$ be a r.v. of law $\mu$. 
For any $\x\in\N^4$, denote 
\[L^\gamma_\m v(\x) = c(\x)+\gamma\E\croc{v(\x-\m+A)},\,\mbox{ for all }\m\in\mathbf M_\x\]
and 
\begin{align*}L^\gamma v(\x) &= \min_{\m\in \mathbf M_\x}L^\gamma_\m v(\x)\\ &= c(\x) +\gamma \min_{\m\in \mathbf M_\x} \E\croc{v(\x-\m+A)}.
\end{align*}
	\end{definition}

\noindent Then we use the principle of dynamic programming, and obtain that for any state $\x\in\N^4$ that is admissible for some policy,  	
%	\noindent On utilise ensuite le principe de la programmation dynamique, résumé par Lawrence C Evans dans son cours d'introduction à la théorie du contrôle optimal : \begin{quote}	
%\noindent 	 Il vaut mieux être intelligent dès le départ, plutôt que d'être stupide avant de devenir intelligent.  \end{quote}
%\medskip 
%\noindent Autrement dit, la politique menée après le premier choix doit être une sous-politique optimale pour l'état résultant de la première décision : 
\begin{align*}
v(\x) & :=  \inf\limits_{\pi \in \scriptsize{\textsc{adm}}} v^\pi_\gamma(\x)\\
&=\inf\limits_{\pi \in \scriptsize{\textsc{adm}}}\E^\pi_\x\croc{\dsum_{n=0}^{+\infty}\gamma^n c(X_n)} \\
&= \inf\limits_{\pi \in \scriptsize{\textsc{adm}}}\E^\pi_\x\croc{c(\x)+\dsum_{n=1}^{+\infty}\gamma^n c(X_n)}  \\
&= c(\x)+\gamma \inf\limits_{\pi \in \scriptsize{\textsc{adm}}}\E^\pi_\x\croc{\dsum_{n=0}^{+\infty}\gamma^nc(X_{n+1})}  \\
&= c(\x)+\gamma \inf\limits_{\stackrel{\by \in \N^4\,:}{\x\to \by\scriptsize{\mbox{ for some }\pi\in\textsc{adm}}}}\esp{v(\by)} \\ 
& = c(\x)+ \gamma \min\limits_{\m\in \mathbf M_\x}\esp{v(\x-\m+A)} = L^\gamma v(\x).
\end{align*}
The mapping $v$ is then called the {\em value function} of our optimization problem, and the argument above shows that $v$ solves a fixed point equation of the Bellman type: for all $\x$ that is admissible for some policy,  
\begin{equation}
\label{eq:bellman}
    v(\x)=L^\gamma v(\x).
\end{equation}
In this paper, we put in evidence a simple policy that is able to achieve a solution to \eqref{eq:bellman} in the case of the `N'-graph.

%\noindent Par conséquent, une fonction de valeur est forcément point fixe de l'opérateur $L^\gamma$. 

%\medskip Voici donc commence s'annonce la suite : \begin{enumerate}
%\item On détermine une fonction de valeurs convenable 
%\footnote{Il se trouve que l'opérateur $L^\gamma$ est contractant dans un espace complet bien choisi. Un tel point fixe est donc unique et peut être approché à vitesse exponentielle par une suite d'itérés.}. 
%\item Pour réaliser une décision dans un état $x$, on choisit un $\tilde{x}\in\mr{Adm}(x)$ qui réalise la borne inférieure de la définition. Cette suite nous donne les décisions à réaliser pour avoir un coût moyen le plus faible possible. 
%\end{enumerate}

%\noindent L'intérêt de cette étude est d'identifier des classes particulières de décisions optimales , pour certains graphes d'appariement. 

\subsection{Threshold-type decision rules}
A decision rule will is to be of {threshold type} in $\ell_2$, with priority for $\ell_1$ and $\ell_3$, if it matches all possible $\ell_1$ and $\ell_3$ edges, before matching a certain amount of $\ell_2$ edges according to a threshold that depends on the difference between the number of remaining items $0$ and $2$. Formally, 
\begin{definition}%[Threshold type decision rule] \ 
Let $(t_i)_{i\in\Z}$ be a family of elements of $\N\cup\{+\infty\}$. The map $u\devers{\N^4}{\N^4}$ is said to be a {\em threshold type decision rule} with priority for $\ell_1$ and $\ell_3$ and thresholds 
$(t_i)_{i\in\Z}$ in $\ell_2$, if it can be written as 
\[u(\x)=\x-m(\x),\quad \x\in\N^4,\]
where 
\[m:\fonccc{\N^4}{\N^4}{\x}{(x_0\wedge x_1)\ell_1 + (x_2\wedge x_{3})\ell_3+ k_{t_{i(\x)}}(\x)\ell_2},\]
%$$m_x=\min(x_0,x_1)\ell_1 + \min(x_2,x_{3})\ell_3+ k_{t_{i(\x)}}(x)\ell_2$$ 
\noindent 
for 
\[\begin{cases}
    k_t(\x) &= \left((\pasc{x_1-x_0})\wedge (x_2-x_3) - t\right)^+;\\[3mm]
    i(\x) &=\pasc{(x_{1}-x_{0})}-(x_2-x_3).
\end{cases}\]
\noindent We denote by $\mathscr T$, the set of threshold decision rules. A policy $\pi$ is said to be {\em of the threshold type}, if it is structured by decision rules of $\mathscr T$.  
%and by $\Pi^\sigma$, the set of policies structured around a single threshold decision rule of $\mathscr T$. %\pasc{Je ne comprends pas cette dernière définition.}
\end{definition}

\subsection{Main result}

\noindent The following result shows that we can minimize the cost in the long run, by applying at any time, a matching policy of the threshold type, conveying the natural idea that it is preferable to match edges $\ell_3$ and $\ell_1$ over edges $\ell_2$, up to a certain state of 
congestion of the buffer. We have the following, 

\begin{theorem}
\label{thm:main}
For the discounted cost problem, under the assumptions of Definition \ref{defi:hypothese} there exists an optimal stationary policy of the threshold type. 
\end{theorem}

\noindent Theorem \ref{thm:main} is proved in Section \ref{sec:proof}. 
\begin{remark}\rm 
It is significant that the above result holds regardless of the fact that the model at hand is instable - a fact that is implied by the first assertion of Theorem 2 in 
\cite{MaiMoy16}, as the `N'-graph is bipartite. It is easy to observe that the quantity in \eqref{eq:defdiscount} is well-defined for all $\pi$ and $\x$, since for all 
$n$, the $n$-th term of the summation is upper-bounded by $\gamma^n \left(\max\{c_0,\cdots,c_3\}(|\x|+n)\right)$. 
\end{remark}
%For this, we first want to convey the idea that 

\section{Proof of Theorem \ref{thm:main}}
\label{sec:proof}
\noindent We will crucially use the following result, 
\begin{theorem}[\cite{puterman2014markov}, 6.11.3]\label{theocentral} \
Suppose that there exists a function $w\devers{\N^4}{\R_+}$ such that
\begin{itemize} 
\item[(i)] For any $\x\in\N^4$, 
$$\sup\limits_{\bz\in\N^4\,:\,\x-\bz\in\N^4} \dfrac{c(\x-\bz)}{w(\x)}<+\infty \quad \text{ and } \quad \sup\limits_{\bz\in\N^4\,:\,\x-\bz\in\N^4} \dfrac{1}{w(\x)}\dsum_{\by} \PP(\by | \x-\bz)w(\by) <+\infty.$$

\item[(ii)] For all $\varepsilon\in[0:1)$, there exists $\eta\in[0:1)$ and $p\in\N\et$ such that for all $p$-tuple $\pi=(u_1,u_2,\ldots, u_p)$ of deterministic Markovian decision rules $u_i\devers{\N^4}{\N^4}$, 
%when we denote by $\PP_\pi$ the $J$-steps transition matrix under the policy $\pi$, and for all $x\in\N^4$,
$$\varepsilon^p \mathbb E^\pi_\x\left[w(X_p)\right] < \eta w(\x).$$
%$$\mu^J \dsum_{y}\PP_\pi (y | x) w(y) < \eta w(x)$$

\item[(iii)] If we denote $$V_w=\acco{v\devers{\N^4}{\R_+}\,:\, \sup\limits_{x\in\N^4} \dfrac{v(\x)}{w(\x)}<\infty},$$ 
\noindent then for all $v\in V_w$ there exists a deterministic Markovian decision rule $u$ such that for all $\x\in\N^4$, 
$$L^\gamma v(\x) = L^\gamma_{u(\x)}v(\x).$$
\end{itemize}

\noindent Suppose also that there exist two sets $\mathscr V$ and $\mathscr D$ such that 

\begin{enumerate}
\item $\mathscr V$ is stable under $L^\gamma$ and under pointwise convergence;
\item If $v\in \mathscr V$, then there exists a deterministic markovian decision rule $u\in \mathscr D$ such that for all $\x \in \N^4,$ 
\[L^\gamma v(\x) = L^\gamma_{u(\x)}v(\x).\]
\end{enumerate}

\noindent Then, there exists an optimal stationary policy $\pi\et$ structured by a single Markovian decision rule $u\et\in \mathscr D$.
\end{theorem}

The strategy is as follows. 
We start by defing a set of test functions $\mathscr V$ that is stable under the action of $L^\gamma$, and such that if $v\in \mathscr V$, we can always find a threshold type decision rule $m\in \mathscr T$ such that \eqref{eq:bellman} holds for all $\x\in\N^4$. 

\subsection{Properties of value functions}

\begin{definition}%[$\mathcal{I}_{\ell_1}$ and $\Ilq$] \
A mapping $v\devers{\N^4}{\R_+}$ is said to be {\em increasing in} $\ell_1$, and we write $v\in\Ilun$ if, for any $\x\in\N^4$ %such that $x_{0}x_{1}>0$ 
we have $$v(\x)\leq v(\x+\ell_1).$$
Likewise, a mapping $v\devers{\N^4}{\R_+}$ is said to be {\em increasing in} $\ell_3$, and we write $v\in\Ilq$ if, for any $\x\in\N^4$ %such that $x_{2}x_{3}>0$ 
we have $$v(\x)\leq v(\x+\ell_3).$$ 
\end{definition}

%\begin{definition}%[$\mathcal{I}_{\ell_1}$ and $\Ilq$] \
%A mapping $v\devers{\N^4}{\R_+}$ is said to be {\em increasing in} $\ell_1$, and we write $v\in\Ilun$ if, for any $\x\in\N^4$ such that $x_{0}x_{1}>0$ we have $$v(\x-\ell_1)\leq v(\x).$$
%Likewise, a mapping $v\devers{\N^4}{\R_+}$ is said to be {\em increasing in} $\ell_3$, and we write $v\in\Ilq$ if, for any $\x\in\N^4$ such that $x_{2}x_{3}>0$ we have $$v(\x-\ell_3)\leq v(\x).$$ 
%\end{definition}

{It is reasonable to think that if we have the choice between matching an $\ell_2$ edge and both an $\ell_1$ and an $\ell_3$ edge, we will rather do the latter. Hence the following definition,     
\begin{definition}%[$\mathcal{I}'$] \ 
We say that $v\devers{\N^4}{\R_+}$ belongs to $\Ip$ if, 
for all $\x\in \N^4$ such that $x_1x_2>0$,  
\[\begin{cases}
 v(\x) & \leq  v(\x+\ell_1-\ell_2)  \\
 v(\x) & \leq  v(\x+\ell_3-\ell_2).
\end{cases}\]
%for all $\x\in (N^*)^4$, 
%\[\begin{cases}
% v(\x-\ell_1+\ell_2) & \leq  v(\x)  \\
% v(\x-\ell_3+\ell_2) & \leq  v(\x).
%\end{cases}\]
\end{definition}
%\noindent Je ne comprends pas le sens concret de cette définition. A reprendre.
}

\noindent Moreover, we want to highlight that it is less and less interesting to keep items of classes \pasc{$1$ or $2$} in line. Hence the following definition,  

\begin{definition}%[$\Cld$] \ 
We say that $v\devers{\N^4}{\R_+}$ is {\em convex in} $\ell_2$ and we write $v\in\Cld$ if, for all state $\x\in\N^4$
such that $x_2\geq x_3-1$ and \pasc{$x_1\geq x_0-1$}, 
$$v(\x+2\ell_2)-v(\x+\ell_2)\geq v(\x+\ell_2)-v(\x).$$
\end{definition}

\subsection{Optimal policies}

\noindent We now show that these properties are already enough to satisfy the required property, namely, that a value function 
solving \eqref{eq:bellman} corresponds to a policy structured by a fixed threshold decision rule. We have the following result, 
%\footnote{Eventually, we will have to take more accurate subset of functions to ensure that stability under $L^\gamma$, but I prefer to to the proof in this order to isolate the properties that are more interesting and pertinent} 

\begin{lemma}\label{lemme35} 
Let $\gamma$, $v\in\Ilun\cap\Ilq\cap\Ip$ and $\x\in\N^4$. 
Then, there exists a matching $\mathbf m\in \mathbf M_\x$ such that $$L^\gamma_{\mathbf m} v(\x) = L^\gamma v(\x),$$ 
and such that $\mathbf m$ matches all possible $\ell_1$ and $\ell_3$ edges in a system in state $\x$. \end{lemma}

\begin{proof}[Proof] 
Let $\x\in \N^4$, and let $\mathbf m^0\in \mathbf M_\x$. %We denote $n_2$ the number of $\ell_2$ in $m$ \footnote{More precisely, $n_2=m_2-m_3$}. 
We let $\m^1$ be the matching we obtain by executing all matches of edges $\ell_1$ and $\ell_3$ that are still possible in the state 
$\x-\m$, that is, 
\[\m^1=\m^0 + (x_0-m^0_0) \wedge (x_1-m^0_1)\ell_1 + (x_2-m^0_2) \wedge (x_3-m^0_3)\ell_3.\] 
Because $v\in\Ilun\cap\Ilq$ and because $\m^1$ can be obtained from  $\m^0$ by successively adding only $\ell_1$ and $\ell_3$ edges, 
we have $$L^\gamma_{\m^1} v(\x) \leq L^\gamma_{\m^0} v(\x).$$
%The matching $\m^1$ has exactly $\min(x_2,x_3)$ $\ell_3$ edges. 
Let $n^1_2$ denotes the number of $\ell_2$ edges in $\m^1$. We are now going to find a matching that is at least as good and which has $\min(x_0,x_1)$ $\ell_1$ edges. To do so, we define $\m^2$ to be the matching obtained from $\m^1$ by transforming as many $\ell_2$ edges in $\ell_1$ edges as possible, 
that is, 
\[\m^2=\m^1 + (n^1_2 \wedge (\pasc{x_0-m^1_0}))(\ell_1-\ell_2).\]
As $\m^2$ can be obtained from $\m^1$ by adding $\ell_1-\ell_2$ a finite number of times, we get that 
$$L^\gamma_{m^2} v(\x) \leq L^\gamma_{m^1} v(\x).$$
Notice that the matching $\m^2$ has exactly $x_0\wedge x_1$ $\ell_1$ edges. Indeed, if it was not so, then we would be in the following alternative:  
\begin{itemize}
    \item[(i)] Either in $\x-\m^2$ there would be a positive number of unmatched $0$-items and a positive number of unmatched $1$-items;
    \item[(ii)] Or, in $\m^2$ there would be a positive number of $1$-items matched with $2$-items.
\end{itemize}
But (i) is in contradiction with the definition of $\m^1$, and (ii) contradicts the definition of $\m^2$, an absurdity. 

\medskip

Now $n^2_2$ denotes the number of $\ell_2$ edges in $\m^2$ (and notice that $n^2_2=0$ if $x_0-m^1_0\ge n^1_2$). 
We now define $\m^3$ to be the matching obtained from $\m^2$ by transforming as many $\ell_2$ edges in $\ell_3$ edges as possible, 
that is, 
\[\m^3=\m^2 + (n^2_2\wedge (\pasc{x_3-m^2_3}))(\ell_3-\ell_2).\]
Again, as $\m^2$ can be obtained from $\m^1$ by adding $\ell_3-\ell_2$ a finite number of times, we get that 
$$L^\gamma_{m^3} v(\x) \leq L^\gamma_{m^2} v(\x).$$
Notice that the matching $\m^3$ also has exactly $x_3\wedge x_2$ $\ell_3$ edges. 
Indeed, if it was not so, we would, again, either find in $\x-\m^3$ a positive number of unmatched $3$-items and a positive number of 
unmatched $2$-items, or find in $\m^2$, a positive number of $2$-items matched with $1$-items, another absurdity given the definitions 
of $\m^1$ and $\m^3$. 

\medskip 

As a conclusion, we have thus shown that for every matching $\m^0$ we can find a matching $\m^3$ which has exactly 
$\min(x_0,x_1)$ $\ell_1$ edges and $\min(x_2,x_3)$ $\ell_3$ edges, and such that 
$$L^\gamma_{m^3} v(\x) \leq L^\gamma_{m^0} v(\x).$$
Since there is a finite number of possible matchings, the proof is complete. 
\end{proof}

\noindent Observe the following result, 

\begin{corollary} 
\label{coro:angels}
For any $0<\gamma<1$, the function set $\Ilun\cap\Ilq\cap\Ip$ is stable by the operator $L^\gamma $.
\end{corollary}

\begin{proof}[Proof]
First, it is easily checked that any linear cost function $c$ satisfying to definition \ref{defi:hypothese} is in particular, an element of 
$\Ilun\cap\Ilq\cap\Ip$. For instance, for any $\x$ such that $x_1x_2>0$, 
\[c(\x+\ell_1-\ell_2)-c(\x)=c(\ell_1)-c(\ell_2)=c_0-c_2 \ge 0,\]
and the other arguments are similar. 
%First, we notice that thanks to \autoref{hypothese}, the cost function belongs to $\Ilun\cap\Ilq\cap\Ip$.
%The proof of the stability is the same as the one for $\Ilun$ and $\Ilq$. \pasc{La stabilité de $\Ilun$ and $\Ilq$ n'est pas encore démontrée.} 
Let $v\in \Ilun\cap\Ilq\cap\Ip$. First let $\x\in\N^4$, and $\bx=\x+\ell_1$. %As $v\in \Ilun$ we have $v(\bx)\geq v(\x)$. %and we aim at showing that we also have $L^\gamma v(\bx)\geq L^\gamma v(\x)$. 
According to Lemma \ref{lemme35}, we can choose a matching $\overline{\m}=(\overline m_0,\overline m_1, \overline m_2, \overline m_3)\in \mathbf M_{\bx}$, 
which matches all possible $\ell_1$ and $\ell_3$ edges, and such that $L^\gamma_{\bm} v(\bx)=L^\gamma v(\bx)$. 
In particular, we have that $\overline m_0\geq1$ and $\overline m_1\geq1$, and so we can define a matching $\m=\overline{\m}-\ell_1\in \mathbf M_{\x}$  
such that $\bx-\overline{\m}=\x-\m$. It follows that 
\begin{align*}
   L^\gamma v(\x) &\le L^\gamma_{\m} v(\x)\\
   & = c(\x) + \gamma \E[v(\x-\m+A)] \\
   & \leq c(\bx) + \gamma\esp{v(\bx-\overline{\m}+A)} \\
   &= L^\gamma_{\bm} v(\bx) = L^\gamma v(\bx).
\end{align*}
As the above holds for all $\x$ we conclude that $L^\gamma v\in\Ilun$, and the same argument shows that $L^\gamma v\in\Ilq$. 

We now prove that $L^\gamma v\in\Ip$. 
%Notice that under the \autoref{hypothese}, we have $c\in\Ip$. We use the same argument for both items of the definition. 
Let $\x\in\N^4$ be such that $x_1x_2>0$, and $\tx=x+\ell_1-\ell_2$. According to Lemma \ref{lemme35}, we can choose a matching $\tm=(\widetilde m_0,\widetilde m_1,\widetilde m_2,\widetilde m_3)\in \mathbf M_{\tx}$, which matches all possible $\ell_1$ and $\ell_3$ edges, and such that 
$L^\gamma_{\tm} v(\tx)=L^\gamma v(\tx)$. Again, we have that $\tilde m_0\ge 1$ and $\tilde m_1 \ge 1$, and so the matching $\m':=\tm+\ell_2-\ell_1$ 
is an element of $\mathbf M_{\x}$ such that $\tx-\tm=\x-\m'$. Thus we have again %(\pasc{il faut vérifier que $c(\x) \le c(\tx)$)}
%In that case, we always have $m_x\in M_x$ 
%\footnote{As a matter of fact, we obtained $\bx$ from $x$, by transforming an item $1$ in an item $2$, and we obtained  $m_{\bx}$ from $m_x$ by transforming an $\ell_1$ edge into an  $\ell_2$. edge}. 
%Therefore,
\begin{align*}
L^\gamma v(\x) &\le L^\gamma_{\m'} v(\x)\\
&= c(\x) + \gamma \E[v(\x-\m'+A)] \\
& \leq c(\tx) + \gamma\E[v(\tx-\tm+A)] \\
& = L^\gamma_{\tm} v(\tx) = L^\gamma v(\tx). 
\end{align*}
We can prove similarly that the same holds true by replacing $\tx$ by $\x+\ell_3-\ell_2$. 
This shows that $L^\gamma v\in\Ip$, which concludes the proof. 
\end{proof}

\begin{proposition}\label{propseuilopti} \ 

\noindent For any $\gamma$ and $v\in \Ilun\cap\Ilq\cap\Cld\cap\Ip$, there exists a decision rule $m\et$ with threshold 
$(t_i)_{i\in\Z}$, such that for all $\x\in\N^4$,$$L^\gamma_{m\et(\x)} v(\x) = L^\gamma v(\x)$$

\end{proposition}

\begin{proof}[Proof]
First, observe that if $i\leq0$, by the convexity of $v$ there exists $t_{i}\in\N \ \cup\acco{+\infty}$ such that the sequence 
$\pare{\esp{v\left(A+|i|\gre_2+ j\ell_2\right)}}_{j\in\N}$ is decreasing on $\crocc{0,t_i}$ and increasing on $\crocc{t_i,+\infty}$. All the same, 
if $i\geq0$, then by convexity of $v$ there exists $t_{i}\in\N \ \cup\acco{+\infty}$ such that the sequence 
$\pare{\esp{v\left(A+i\gre_1+ j\ell_2 \right)}}_{j\in\N}$ is decreasing on $\crocc{0,t_i}$ and increasing on $\crocc{t_i,+\infty}$. 
We call $m\et$ the decision rule with thresholds $(t_i)_{i\in\Z}$. %\pasc{Ici, il y a une confusion entre matching et decision rule.}

\medskip 

Now let $\x\in\N^4$. According to Lemma \ref{lemme35}, there exists $\m\in \mathbf M_\x$ that realizes the minimum 
$L^\gamma v(\x)$, and matches all the possible $\ell_1$ and $\ell_3$ edges of $\x$. We call $k_\x$ the number of $\ell_2$ edges in $\m$. 
We are in the following alternative: 
\begin{enumerate}
\item If $x_0\leq x_1$ or $x_2\leq x_3$, then $k_\x=k_{t_{i(\x)}}(\x)=0$ and $\m=m\et(\x)$. We thus have 
$$L^\gamma_{m\et(\x)} v(\x) = L^\gamma_{\m} v(\x) = L^\gamma v(\x).$$ 
\item Otherwise, denoting $j(\x)=(x_1-x_0)\wedge (x_2-x_3)$, we have that 
\begin{align*}
  L^\gamma_{m\et(\x)}v(\x) &= c(\x)+\esp{v\left(A+(x_1-x_0)\gre_1+(x_2-x_3)\gre_2-k_{t_{i(\x)}}(\x)\ell_2\right)}\\
  &=c(\x)+\esp{v\left(A+\left(x_1-x_0-j(\x)\right)^+\gre_1+\left(x_2-x_3-j(\x)\right)^+\gre_2+\left(j(\x)-k_{t_{i(\x)}}(\x)\right)\ell_2\right)}. 
\end{align*}
%We show that in each possible case, $$L^\gamma_{m\et(\x)}v(x)\leq L^\gamma_{m_x}v(x).$$ 
\noindent We are in the following sub-alternative: 
\begin{itemize}
\item[2a)] If $i(\x)\le 0,$ 
%$x_1-x_0\leq x_2-x_3$, 
then 
we get that 
\begin{equation*}
  L^\gamma_{m\et(\x)}v(\x) %&= c(\x)+\esp{v\left(A+(x_1-x_0)\gre_1+(x_2-x_3)\gre_2-k_{t_{i(\x)}}\ell_2\right)}\\
  =c(\x)+\esp{v\left(A+|i(\x)|\gre_2+\left(j(\x)-k_{t_{i(\x)}}(\x)\right)\ell_2\right)}. 
\end{equation*}
and we also have 
\[L^\gamma_{\m} v(\x) = c(\x) + \esp{v\left(A+ i(\x)\gre_2+(j(\x)-k_\x)\ell_2\right)}.\]

\noindent We show that, in all cases, 
\begin{equation}
    \label{eq:partita6}
    L^\gamma_{m\et(\x)}v(\x)\le L^\gamma_{\m} v(\x). 
\end{equation}
\begin{itemize}
\item If $t_{i(\x)}=+\infty$, we get that $k_{t_{i(\x)}}(\x)=0\leq k_\x$. So $j(\x)-k_{t_{i(\x)}}(\x)\geq j(\x)-k_\x$. But in that case, 
the sequence $\pare{\esp{v\left(A+|i(x)|\gre_2+j\ell_2\right)}}_j$ is decreasing, which implies \eqref{eq:partita6}.  
\item If $t_{i(\x)}=0$, then the sequence $\pare{\E[v\left(A+|i(\x)|\gre_2+j\ell_2\right)]}_j$ is increasing and 
\[j(\x)-k_{t_{i(\x)}}(\x)=0\leq j(\x)-k_\x,\] implying again \eqref{eq:partita6}. 
\item Otherwise, the sequence $\pare{\E[v\left(A+|i(\x)|\gre_2+j\ell_2\right)]}_j$ is decreasing on 
$\crocc{0,t_{i(\x)}}$ and increasing on $\crocc{t_{i(\x)},+\infty}$. In that case, 
\begin{itemize}
\item If $j(\x)\leq t_{i(\x)}$, then $k_{t_{i(\x)}}(\x)=0$ and 
thus $t_i(\x)\ge j(\x)-k_{t_{i(\x)}}(\x)\ge j(\x)-k_\x$. 
%$j(\x)-k_{t_{i(\x)}}(x)\leq j(\x)-k_\x\leq t_{i(\x)}$. 
Again, the sequence $\pare{\E[v\left(A+|i(\x)|\gre_2+j\ell_2\right)]}_j$ 
is decreasing on $\crocc{0,t_{i(\x)}}$, implying \eqref{eq:partita6}. 
\item If $j(\x)\geq t_{i(\x)}$, then $j(\x)-k_{t_{i(\x)}}(x)=t_{i(\x)}$ and \eqref{eq:partita6} follows from the fact that the sequence 
$\pare{\E[v\left(A+|i(\x)|\gre_2+j\ell_2\right)]}_j$ reaches its minimum at $j=t_{i(\x)}$. 
\end{itemize}
\end{itemize}

\item[2b)] If $i(\x)>0$, then $j(\x)=x_2-x_3$ and we have 
\begin{align*}
  L^\gamma_{m\et(\x)}v(\x) %&= c(\x)+\esp{v\left(A+(x_1-x_0)\gre_1+(x_2-x_3)\gre_2-k_{t_{i(\x)}}\ell_2\right)}\\
  &=c(\x)+\esp{v\left(A+(x_1-x_0-j(\x))\gre_1+\left(j(\x)-k_{t_{i(\x)}}(\x)\right)\ell_2\right)}\\
  &=c(\x)+\esp{v\left(A+i(\x)\gre_1+\left(j(\x)-k_{t_{i(\x)}}(\x)\right)\ell_2\right)};\\
L^\gamma_{\m} v(\x) &= c(\x) + \esp{v\left(A+ i(\x)\gre_1+(j(\x)-k_\x)\ell_2\right)}.
\end{align*}
%rewrite the expression above  regrouping $s_0$ and $s_2$ items to create some $\ell_2$ edges : 
%\centers{$L^\gamma_{m\et(\x)}v(x) = c(x) + \E[v(A + (j(\x)-k_{t_{i(\x)}})\ell_2 + \underbrace{(x_0-x_1-j(\x))}_{=i(x)}s_0] $}\medskip 
%\leftcenters{et}{$L^\gamma_{m_x} v(x) = c(x) + \E[v(A+(j(\x)-k_\x)\ell_2 + i(x)s_2]$}
\noindent By the symmetry between $\ell_1$ and $\ell_2$, we then conclude that \eqref{eq:partita6} holds, just like in sub-case 2a). 
\end{itemize}
\end{enumerate}
\end{proof}

\subsection{Value function property conservation}

\noindent In this section, we aim at finding a subset of $\Ilun\cap\Ilq\cap\Cld\cap\Ip$ that is stable under $L^\gamma$ for well chosen values 
of $\gamma$. For this we introduce the following two functional spaces, 
%so that we can use [6.11.3] in \cite{puterman2014markov}. Here are two definitions that are less intuitive those stated before, but that will ensure stability.

\begin{definition} \label{defi:B} %[$\BB$] \ 
We denote by $\BB$, the set of mappings $v\devers{\N^4}{\R_+}$ such that, for all $\bb\in\N^4$,
\begin{enumerate}
\item[(i)] If $\bb=\pasc{(1,0,\beta,0)}+\gre_i$ \pasca{for $i\in\{0,2,3\}$ and $\beta\ge 0$,}
\begin{equation}
    \label{eq:says1}
v\left(\bb+2\ell_2-\ell_1\right)-v\left(\bb+\ell_2-\ell_1\right)  \geq  v\left(\bb+\ell_2-\ell_1\right)-v\left(\bb\right) ;
\end{equation}

\item[(ii)] If $\bb=\pasc{(0,\beta,0,1)}+\gre_i$ \pasca{for $i\in\{0,1,3\}$ and $\beta\ge 0$,} 
\begin{equation}
    \label{eq:says2}
v\left(\bb+2\ell_2-\ell_3\right)-v\left(\bb+\ell_2-\ell_3\right)  \geq  v\left(\bb+\ell_2-\ell_3\right)-v\left(\bb\right);
\end{equation}

\item[(iii)] If $\bb=\pasc{(1,0,0,1)}+\gre_i$ for $i\in\llbracket 0,3 \rrbracket$, 
\[v\left(\bb+2\ell_2-\ell_1-\ell_3\right)-v\left(\bb+\ell_2-\ell_1-\ell_3\right)  \geq  v\left(\bb+\ell_2-\ell_1-\ell_3\right)-v(\bb).\]
\end{enumerate} 
\end{definition}

\noindent Observe the following,
\begin{lemma}
    \label{lemma:B}
    For any $\beta\ge 0$, 
    \begin{enumerate}
        \item[(i)] \eqref{eq:says1} holds for any $v\in\Ilun\cap\Cld\cap\BB$ and $\bb={(1,0,\beta,0)}+\gre_i$, $i\in\llbracket 0,3 \rrbracket$; 
        \item[(ii)] \eqref{eq:says2} holds for any $v\in\Ilq\cap\Cld\cap\BB$ and $\bb={(0,\beta,0,1)}+\gre_i$, $i\in\llbracket 0,3 \rrbracket$. 
    \end{enumerate} 
\end{lemma}
\pasca{
\begin{proof}[Proof]
    \begin{enumerate}
        \item[(i)] As $v\in \BB$, by the very assertion (i) in definition \ref{defi:B} it is sufficient to check that \eqref{eq:says1} holds for $\bb$ of the form 
        $\bb=(1,1,\beta,0)$. But this is true since, as $v\in\Ilun\cap\Cld$, we then have 
        \[v\left(\bb+\ell_2-\ell_1\right)-v\left(\bb\right) \le v\left(\bb+\ell_2-\ell_1\right)-v\left(\bb-\ell_1\right)\\
            \le v\left(\bb+2\ell_2-\ell_1\right)-v\left(\bb+\ell_2-\ell_1\right).\]
        \item[(ii)] Likewise, it is enough to show that \eqref{eq:says2} holds for $\bb$ of the form 
        $\bb=(0,\beta,1,1)$. But then, as $v\in\Ilq\cap\Cld$ we get 
        \[v\left(\bb+\ell_2-\ell_3\right)-v\left(\bb\right) \le v\left(\bb+\ell_2-\ell_3\right)-v\left(\bb-\ell_3\right)\\
            \le v\left(\bb+2\ell_2-\ell_3\right)-v\left(\bb+\ell_2-\ell_3\right).\]
    \end{enumerate}
\end{proof}
}

\begin{definition}
\label{defi:C'}
We denote by $\Cp$, the set of mappings $v\devers{\N^4}{\R_+}$ such that, for all $\cc\in\N^4$, 
\begin{enumerate}
\item[(i)] If $\cc=\pasc{(\alpha,0,\beta,0)}+\gre_i$ for $\alpha\geq2$, $\beta\geq0$ and $i\in\llbracket 0,3 \rrbracket$, 
\begin{equation*}
    %\label{eq:says3}
    v(\cc+2(\ell_2-\ell_1))-v(\cc+\ell_2-\ell_1)\geq v(\cc+\ell_2-\ell_1)-v(\cc); 
\end{equation*}
\item[(ii)] If $\cc=\pasc{(\alpha,0,0,1)}+\gre_i$ for $\alpha\geq2$ and $i\in\llbracket 0,3 \rrbracket$,  
\[v(\cc+2(\ell_2-\ell_1)-\ell_3)-v(\cc+\ell_2-\ell_1-\ell_3) \geq v(\cc+\ell_2-\ell_1-\ell_3)- v(\cc);\]
\item[(iii)] If $\cc=\pasc{(0,\beta,0,\alpha)}+\gre_i$ for $\alpha\geq2$, $\beta\geq0$ and $i\in\llbracket 0,3 \rrbracket$, 
\[v(\cc+2(\ell_2-\ell_3))-v(\cc+\ell_2-\ell_3)\geq v(\cc+\ell_2-\ell_3)-v(\cc);\]
\item[(iv)] If $\cc=\pasc{(1,0,0,\alpha)}+\gre_i$ for $\alpha\geq2$ and $i\in\llbracket 0,3 \rrbracket$, 
\[v(\cc+2(\ell_2-\ell_3)-\ell_1)-v(\cc+\ell_2-\ell_3-\ell_1)\geq v(\cc+\ell_2-\ell_3-\ell_1) - v(\cc);\]
\item[(v)] If $\cc=\pasc{(\alpha,0,0,\beta)}+\gre_i$ for $\alpha,\beta\geq2$ and $i\in\llbracket 0,3 \rrbracket$, 
\[v(\cc+2(\ell_2-\ell_1-\ell_3))-v(\cc+\ell_2-\ell_1-\ell_3)\geq v(\cc+\ell_2-\ell_1-\ell_3)-v(\cc);\]
\end{enumerate}
\end{definition}

%\pasca{\begin{lemma}
%    \label{lemma:C'}
%    For any $\beta\ge 0$, 
 %   \begin{enumerate}[label=(\roman*)]
 %       \item \eqref{eq:says3} holds for any $v\in\Ilun\cap\Cld\cap\Cp$ and $\bb={(\alpha,0,\beta,0)}+\gre_i$, $\alpha\ge 2$, $i\in\llbracket 0,3 \rrbracket$;  
%    \end{enumerate} 
%\end{lemma}}
%\pasca{\begin{proof}[Proof]
%\begin{enumerate}[label=(\roman*)]
%        \item From (i) in definition \ref{defi:C'}, we only need to show that \eqref{eq:says3} holds for $\cc$ of the form $(\alpha,1,\beta,0)$, $\alpha\ge 2$. 
 %       But then, as $v$ also belongs to $\Ilun\cap\Cld$, 
 %       \begin{align*}
 %       v(\cc+\ell_2-\ell_1)-v(\cc) &\le v(\cc+\ell_2-\ell_1)-v(\cc-\ell_1)\\
 %                                   &\le v(\cc+2\ell_2-\ell_1)-v(\cc+\ell_2-\ell_1)\\
 %                                   &\le v(\cc+2\ell_2-\ell_1)-v(\cc+\ell_2-\ell_1)
  %      \end{align*}
 %   \end{enumerate}
%\end{proof}
%}

\noindent We proceed with our key technical result, namely, the study of the stability under $L^\gamma$. 

\begin{proposition} 
\label{prop:stable}
For any $0<\gamma<1$, the set \[\mathscr V := \Ilun\cap\Ilq\cap\Ip\cap\Cld\cap\mathcal{B}\cap\Cp,\] 
is stable by the operator $L^\gamma$.  
\end{proposition}

\begin{proof}[Proof]
Again, it is immediate to observe that any admissible linear cost function is itself an element of $\mathscr V$. 
%The stability of the first three factors has been shown in the \autoref{lemme35}.
Let $v\in \mathscr V$. We know from Corollary \ref{coro:angels} that $L^\gamma v$ belongs to $\Ilun\cap\Ilq\cap\Ip$. 
We also let $m\et$ be a decision rule of threshold $(t_i)_{i\in\Z}$ for $v$, satisfying the properties of Proposition \ref{propseuilopti}. We then show successively that $L^\gamma$ also is an element of $\Cld$, $\mathcal{B}$ and $\Cp$. 

\paragraph{\underline{Step I: $L^\gamma v \in \Cld$}.} Let $\x\in\N^4$ be such that $x_1\geq x_0-1$ and $x_2\geq x_3-1$. 
Denote $\ol{\x}=\x+\ell_2$, $\oll{\x}=\x+2\ell_2$ and 
the r.v. $\bb=\x-m\et(\x)+A$. Then, observing that $i(\x)=i(\ol{\x})=i(\oll{\x})$ by definition, denote $$t\et = t_{i(\x)}=t_{i(\ol{\x})}=t_{i(\oll{\x})}.$$

\paragraph{First case: $x_1\geq x_0$ and $x_2\geq x_3$.} 
%In that case, $b+A$ has its value in the domain of the definition of $\Cld$ . Therefore, we will use the fact that $v\in\Cld$. 
We are in the following alternative: 
\begin{itemize}
\item[(1a)] If $k_{t\et}(\ol{\x})>0$, then $k_{t\et}(\oll{\x})>0$ and in that case we get 
$$k_{t\et}(\oll{\x})=k_{t\et}(\ol{\x})+1=k_{t\et}(\x)+2$$ and so 
$$\x-m\et(\x)=\ol{\x}-m\et\left(\ol{\x}\right)=\oll{\x}-m\et\left(\oll{\x}\right).$$
Therefore, as $c\in\Cld$ we get  
\begin{align*}
L^\gamma v(\ol{\x})-L^\gamma v(\x) 
= L^\gamma_{m\et(\ol{\x})} v(\ol{\x})-L^\gamma_{m\et(\x)} v(\x)                                  
&=  c(\ol{\x})-c(\x)+\gamma\esp{v\left(\bb\right)-v\left(\bb\right)}\\
%&=  c(\ol{\x})-c(\x)\\
&\leq c(\oll{\x})-c(\ol{\x})+\gamma\esp{v\left(\bb\right)-v\left(\bb\right)}\\
%&= c(\oll{\x})-c(\ol{\x})+\gamma\esp{v\left(\oll{\x}-m\et(\oll{\x})+A\right)-v\left(\ol{\x}-m\et\left(\ol{\x}\right)+A\right)}\\
&= L^\gamma_{m\et(\oll{\x})} v(\oll{\x}) - L^\gamma_{m\et(\ol{\x})} v(\ol{\x})
= L^\gamma v(\oll{\x}) - L^\gamma v(\ol{\x}).
\end{align*}

\item[(1b)] If $k_{t\et}(\ol{\x})=0$ and $k_{t\et}(\oll{\x})>0$, then $$k_{t\et}(\oll{\x})=k_{t\et}(\ol{\x})+1=k_{t\et}(\x)+1.$$ 
In that case, we have  
%$m\et\left(\ol{\x}\right)=m\et(\x)$ 
%\footnote{\label{memecouplage}As $x_2\geq x_3$, there really is no edge to add to the matching anymore}
\[\ol{\x}-m\et\left(\ol{\x}\right)=x-m\et(\x)+\ell_2 \quad\text{and} \quad \oll{\x}-m\et\left(\oll{\x}\right)=\ol{\x}-m\et\left(\ol{\x}\right).\]
\noindent Consequently, as $m\et\left(\ol{\x}\right)+\ell_2\in \mathbf M_{\ol{\x}}$ we get that 
\begin{align*}
L^\gamma v(\ol{\x})-L^\gamma v(\x) %& =  c(\ol{\x})-c(\x) + \gamma\esp{v\left(\x-m\et(\x)+\ell_2+A\right)-v\left(\x-m\et(\x)+A\right)} \\
 %&= c(\ol{\x})-c(\x)  + \gamma \esp{v\left(\ol{\x}-m\et\left(\ol{\x}\right)+A\right) -v\left(\ol{\x}-\left(m\et\left(\ol{\x}\right)+\ell_2\right)+A\right)}\\
 &= c(\ol{\x})-c(\x) + L^\gamma_{m\et\left(\ol{\x}\right)}v(\ol{\x})-L^\gamma_{m\et\left(\ol{\x}\right)+\ell_2}v(\ol{\x})\\
 &\le c(\ol{\x})-c(\x)\\
 & \leq  c(\oll{\x})-c(\ol{\x}) 
 %&= c(\oll{\x})-c(\ol{\x})+\gamma\E[v(b+\ell_2+A)-v(b+\ell_2+A)]  \\
 = L^\gamma v(\oll{\x})-L^\gamma v(\ol{\x}).
\end{align*}
\item[(1c)] If $k_{t\et}(\ol{\x})=k_{t_2}(\oll{\x})=0$, reasoning as for the previous case we get that %$m\et(\x)=m\et\left(\ol{\x}\right)=m\et\left(\oll{\x}\right)$ 
%and thus  
\[\oll{\x}-m\et\left(\oll{\x}\right)=\ol{\x}-m\et\left(\ol{\x}\right)+\ell_2 =\x-m\et(\x)+2\ell_2,\]
implying, as $v$ belongs to $\Cld$, that 
\begin{align*}
 L^\gamma v(\ol{\x}) - L^\gamma v(\x) &= c(\ol{\x})-c(\x) + \gamma \esp{v\left(\bb+\ell_2\right)-v\left(\bb\right)} \\
 & \leq  c(\oll{\x})-c(\ol{\x}) + \gamma\esp{v\left(\bb+2\ell_2\right)-v\left(\bb+\ell_2\right)}  
 = L^\gamma v(\oll{\x})-L^\gamma v(\ol{\x}).
\end{align*}
\end{itemize}

\paragraph{Second case: \pasc{$x_1=x_0-1$} and $x_2\geq x_3$.} In that case we have $\bb=\pasc{(1,0,\beta,0})+\gre_i$ with $i\in\llbracket 0,3 \rrbracket$, and $m\et(\bx)=m\et(\x)+\ell_1$, so that 
\begin{equation}
\label{eq:angels1}
L^\gamma v(\ol{\x})-L^\gamma v(\x)  =  c(\ol{\x})-c(\x) + \gamma \esp{v\left(\bb+\ell_2 - \ell_1\right)-v\left(\bb\right)}.
\end{equation}
%\[\ol{\x}-m\et\left(\ol{\x}\right)=\x-m\et(\x)+\ell_2-\ell_1.\] 
Then, 
\begin{itemize}
\item[(2a)] If $k_{t\et}(\oll{\x})>0$, we also have $m\et\left(\oll{\x}\right)=m\et\left(\ol{\x}\right)+\ell_2$, hence 
$\oll{\x}-m\et\left(\oll{\x}\right)=\ol{\x}-m\et\left(\ol{\x}\right)$. Thus, as $v\in \Ip$ it follows from \eqref{eq:angels1} that 
\begin{equation*}
    L^\gamma v(\ol{\x})-L^\gamma v(\x) %& =  c(\ol{\x})-c(\x) + \gamma \esp{v\left(\bb+\ell_2 - \ell_1\right)-v\left(\bb\right)} \\
    \le c(\ol{\x})-c(\x)
    \le c(\oll{\x})-c(\ol{\x}) =  L^\gamma v(\oll{\x}) - L^\gamma v(\ol{\x}). 
\end{equation*}

\item[(2b)] If $k_{t\et}(\oll{\x})=0$, then $m\et\left(\oll{\x}\right)=m\et\left(\ol{\x}\right)$
and, as $c\in\Cld$ and $v\in \BB$, \eqref{eq:angels1} entails  
\begin{equation*}
L^\gamma v(\ol{\x})-L^\gamma v(\x) %& =  c(\ol{\x}) - c(\x) + \gamma\esp{v(\bb+\ell_2-\ell_1)-v(\bb)} \\
%\leq  c(\oll{\x})-c(\ol{\x})+ \gamma\esp{v(\bb+\ell_2-\ell_1+A) - v(\bb+A)} \\
\leq  c(\oll{\x})-c(\ol{\x})+\gamma\esp{v(\bb+2\ell_2-\ell_1)-v(\bb+\ell_2-\ell_1)} 
= L^\gamma v(\oll{\x}) - L^\gamma v(\ol{\x}). 
\end{equation*}
\end{itemize}

\paragraph{Third case: \pasc{$x_1\geq x_0$} and $x_2=x_3-1$.} This case is symmetrical to the second case, replacing $\ell_1$ by $\ell_3$, and using now assertion (ii) of definition \ref{defi:B}). 
%Then, $\bb=(0,\beta,0,1)+\gre_i,$ for $\beta\geq0$ and $i\in\llbracket 0,3 \rrbracket$, and $m\et(\bx)=m\et(\x)+\ell_3$, so \centers{$\ol{\x}-m\et\left(\ol{\x}\right)=x-m\et(\x)+\ell_2-\ell_3$}
%\begin{itemize}
%\item If $k_{t\et}(\oll{\x})>0$, we also have $m\et\left(\oll{\x}\right)=m\et\left(\ol{\x}\right)+\ell_2$, so $\oll{\x}-m\et\left(\oll{\x}\right)=\ol{\x}-m\et\left(\ol{\x}\right)$. Thus, 

%\begin{calculs}
%& L^\gamma v(\ol{\x})-L^\gamma v(x) & = & c(\ol{\x})-c(x) + \gamma \E[v(b+\ell_2 - \ell_3+A)-v(b+A)] & \\[1mm] & &\leq & c(\oll{\x})-c(\ol{\x})+\gamma\E[v(b+A)-%v(b+A)] & $v\in\Ip$  \\[1mm] & & = & L^\gamma v(\oll{\x}) - L^\gamma v(\ol{\x}) &
%\end{calculs}

%\item If $k_{t\et}(\oll{\x})=0$, then again, $m\et\left(\ol{\x}\right)=m\et(\x)+\ell_3$, and this time $m\et\left(\oll{\x}\right)=m\et\left(\ol{\x}\right)$

%Therefore,

%\begin{calculs}
%& L^\gamma v(\ol{\x})-L^\gamma v(x) & = & c(\ol{\x}) - c(x) + \gamma\E[v(b+\ell_2-\ell_3+A)-v(b+A)] & \\[1mm] & &\leq & c(\oll{\x})-c(\ol{\x})+ \gamma\E[v(b+\ell_2-\ell_3+A) - v(b+A)] & \\[1mm] & &\leq & c(\oll{\x})-c(\ol{\x})+\gamma\E[v(b+2\ell_2-\ell_3+A)-v(b+\ell_2-\ell_3+A)]  & \ \ ($c\in \Cld$ et $v\in\mathcal{B}$) \\[1mm] & &=& L^\gamma v(\oll{\x}) - L^\gamma v(\ol{\x}) &
%\end{calculs}
%\end{itemize}

\paragraph{Fourth case: \pasc{$x_1=x_0-1$} and $x_2=x_3-1$.} Then $\bb$ is of the form $\pasc{(1,0,0,1)}+\gre_i$ for some $i\in\llbracket 0,3 \rrbracket$.  
Then, $m\et(\bx)=m\et(\x)+\ell_1+\ell_3$, and as $c\in \Cld$ we have that 
\begin{align}
L^\gamma v(\ol{\x})-L^\gamma v(\x) &= c(\ol{\x})-c(\x) + \gamma \esp{v(\bb+\ell_2 - \ell_1-\ell_3)-v(\bb)}\notag\\
&\le c(\oll{\x})-c(\ol{\x}) + \gamma \esp{v(\bb+\ell_2 - \ell_1-\ell_3)-v(\bb)}.\label{eq:angels2}
\end{align} 
%in every case, $m\et(\bx)=m\et(\x)+\ell_1+\ell_3$, so \centers{$\ol{\x}-m\et\left(\ol{\x}\right)=x-m\et(\x)+\ell_2-\ell_1-\ell_3$}
We have the following alternative: 
\begin{itemize}
\item[(4a)] If $k_{t\et}(\oll{\x})>0$, then we have $m\et\left(\oll{\x}\right)=m\et\left(\ol{\x}\right)+\ell_2$, hence $\oll{\x}-m\et\left(\oll{\x}\right)=\ol{\x}-m\et\left(\ol{\x}\right)$. Thus it follows from \eqref{eq:angels2} and the fact that $v\in\Ip\cap\Ilun$, that 
\[L^\gamma v(\ol{\x})-L^\gamma v(\x) \leq  c(\oll{\x})-c(\ol{\x})+\gamma\esp{v(\bb)-v(\bb)} =  L^\gamma v(\oll{\x}) - L^\gamma v(\ol{\x}).\]

\item[(4b)] If $k_{t\et}(\oll{\x})=0$, then $m\et\left(\oll{\x}\right)=m\et\left(\ol{\x}\right)$ and hence, as $v\in\BB$, \eqref{eq:angels2} implies that 
\begin{align*}
L^\gamma v(\ol{\x})-L^\gamma v(\x) &\leq  c(\oll{\x})-c(\ol{\x})+\gamma\esp{v(\bb+2\ell_2-\ell_1-\ell_3)-v(\bb+\ell_2-\ell_1-\ell_3)}\\
&= L^\gamma v(\oll{\x}) - L^\gamma v(\ol{\x}),
\end{align*}
which concludes the proof in the fourth case. 
\end{itemize}
To summarize, in all cases we get that 
\[L^\gamma v(\ol{\x})-L^\gamma v(\x) \le L^\gamma v(\oll{\x}) - L^\gamma v(\ol{\x}).\]
As this is true for all $\x\in\N^4$ such that $x_1\geq x_0-1$ and $x_2\geq x_3-1$, we conclude that $L^\gamma v\in \Cld$. 

\paragraph{\underline{Step II: $L^\gamma v \in \BB$}.}

\paragraph{Assertion (i).} We first check assertion (i) of definition \ref{defi:B}. So we suppose that $\bb=(1,0,\beta,0)+\gre_i$ for $\beta\ge 0$ and 
$i\in \{0,2,3\}$. Throughout this part of the proof we denote $\olb=\bb+\ell_2-\ell_1$ and $\ollb=\bb+\pasc{2\ell_2-\ell_1}$, and let $t\et = t_{i\left(\olb\right)}=t_{i\left(\ollb\right)}$. We show that $L^\gamma(\olb)-L^\gamma(\bb)\le L^\gamma\left(\ollb\right)-L^\gamma\left(\olb\right)$ for each $i\in\{0,2,3\}$:   

\begin{itemize}
%\item If $i=0$, then $b=(1,1,\beta,0)$, $b'=(1,0,\beta+1,0)$ and $b''=(2,0,\beta+2,0)$, so

%\begin{calculs}
%& L^\gamma v(b) &=& c(b) + \gamma\E[v(b-\ell_1+A)] & \\ & L^\gamma v(b') & = & c(b') + \gamma \E[v(b+\ell_2-\ell_1-k_{t\et}(b')\ell_2+A)] & \\& L^\gamma v(b'') %& = & c(b'') + \gamma \E[v(b+2\ell_2-\ell_1-k_{t\et}(b'')\ell_2+A)] &
%\end{calculs}
%
%\begin{itemize}
%\item If $k_{t\et}(b'')=0$, then we use the fact that the random variable $b-\ell_1+A$ takes its values in the domain of the definition of $\Cld$. 
%
%\item If $k_{t\et}(b'')=1$, then $k_{t\et}(b')=0$, but $m\et_{b'}+\ell_2\in M_{b'}$. Consequently, 
%
%\begin{calculs}
%& L^\gamma(b')-L^\gamma (b) & = & c(b') - c(b) + \underbrace{\gamma \E[v(b+\ell_2-\ell_1+A)-v(b-\ell_1+A)]}_{= L^\gamma_{m\et_{b'}}v(b') - %L^\gamma_{m\et_{b'}+\ell_2} v (b')\leq 0} \\[4mm] & &\leq& c(b'')-c(b') & \\ & &=& L^\gamma v(b'') - L^\gamma v(b') &  
%
%\end{calculs} 
%
%\item If $k_{t\et}(b'')=0$, then $$L^\gamma v(b') - L^\gamma v(b) = c(b')-c(b) \leq c(b'') - c(b') = L^\gamma v (b'') - L^\gamma v (b')$$
%\end{itemize}

\item[(ia)] If $i=0$, then $\bb=(2,0,\beta,0)$, $\olb=(1,0,\beta+1,0)$ and $\ollb=\pasc{(1,1,\beta+2,0)}$, implying that $m\et\left(\bb\right)=m\et\left(\olb\right)=\mathbf 0$ and $m\et\left(\ollb\right)=\ell_1.$ 
So as $c\in\BB$, it follows from (i) of definition \ref{defi:C'} that   
\begin{align*}
 L^\gamma(\olb)-L^\gamma(\bb) &=c\left(\olb\right)-c(\bb)+\gamma\esp{v(\bb+A+\ell_2-\ell_1)-v(\bb+A)}\\
                      &\le c\left(\ollb\right)-c\left(\olb\right) +\gamma \esp{v(\bb+A+2(\ell_2-\ell_1))-v(\bb+\ell_2-\ell_1+A)}\\
                      %&= c\left(\ollb\right)-c\left(\olb\right) +\gamma \esp{v\left(\ollb -\ell_1+A\right)-v\left(\olb+A\right)}\\
                      &=L^\gamma\left(\ollb\right)-L^\gamma\left(\olb\right).
\end{align*}

\item[(ib)] If $i=2$, then $\bb=\pasc{(1,0,\beta+1,0)}$, $\olb=(0,0,\beta+2,0)$ and $\ollb=\pasc{(0,1,\beta+3,0)}$, in a way that 
$m\et\left(\bb\right)=m\et\left(\olb\right)=\mathbf 0$ and $m\et\left(\ollb\right)=k_{t\et}\left(\ollb\right)\ell_2.$ 
Then, 
\begin{itemize}
    \item If $k_{t\et}\left(\ollb\right)=0$, then in view of (i) of Lemma \ref{lemma:B} we have  
    \begin{align*}
 L^\gamma(\olb)-L^\gamma(\bb) &=c\left(\olb\right)-c(\bb)+\gamma\esp{v(\bb+A+\ell_2-\ell_1)-v(\bb+A)}\\
                      &\le c\left(\ollb\right)-c\left(\olb\right) +\gamma \esp{v(\bb+A+2\ell_2-\ell_1)-v(\bb+\ell_2-\ell_1+A)}\\
                      &=L^\gamma\left(\ollb\right)-L^\gamma\left(\olb\right).
\end{align*}
    %thus, just as above, 
     \item If $k_{t\et}\left(\ollb\right)=1$, then, as $v\in\Ip$, 
     \begin{align*}
 L^\gamma(\olb)-L^\gamma(\bb) &=c\left(\olb\right)-c(\bb)+\gamma\esp{v(\bb+A+\ell_2-\ell_1)-v(\bb+A)}\\
                      &\le c\left(\olb\right)-c(\bb)\\
                      &\le c\left(\ollb\right)-c\left(\olb\right)=L^\gamma\left(\ollb\right)-L^\gamma\left(\olb\right).
\end{align*}
\end{itemize}

\item[(ic)] If $i=3$, then \pasc{$\bb=(1,0,\beta,1)$, $\olb=(0,0,\beta+1,1)$ and $\ollb=(0,1,\beta+2,1)$}. 
Then we distinguish three cases : 

\begin{itemize}
\item If $\beta\geq1$ and $k_{t\et}\left(\ollb\right)=0$, then applying (i) of Lemma \ref{lemma:B} to $\bb-\ell_3+A$ we obtain 
\begin{align*}
 L^\gamma(\olb)-L^\gamma(\bb) &=c\left(\olb\right)-c(\bb)+\gamma\esp{v(\bb+\ell_2-\ell_1-\ell_3+A)-v(\bb-\ell_3+A)}\\
                      &\le c\left(\ollb\right)-c\left(\olb\right) +\gamma \esp{v(\bb+2\ell_2-\ell_1-\ell_3+A)-v(\bb+\ell_2-\ell_1-\ell_3+A)}\\
                      &=L^\gamma\left(\ollb\right)-L^\gamma\left(\olb\right).
\end{align*}
\item If $\beta\geq1$ and $k_{t\et}\left(\ollb\right)=1$, then as $v\in\Ip$ we get 
\begin{align*}
 L^\gamma(\olb)-L^\gamma(\bb) &=c\left(\olb\right)-c(\bb)+\gamma\esp{v(\bb+\ell_2-\ell_1-\ell_3+A)-v(\bb-\ell_3+A)}\\
                      &\le c\left(\olb\right)-c(\bb)\\
                      &\le c\left(\ollb\right)-c\left(\olb\right)\\
                      %&\le c\left(\ollb\right)-c\left(\olb\right)+\gamma \esp{v(\bb+\ell_2-\ell_1-\ell_3+A)-v(\bb+\ell_2-\ell_1-\ell_3+A)}\\
                      &= c\left(\ollb\right)-c\left(\olb\right)+\gamma \esp{v\left(\ollb-\ell_2-\ell_3+A\right)-v(\olb-\ell_3+A)}\\
                      &=L^\gamma\left(\ollb\right)-L^\gamma\left(\olb\right).
\end{align*}
\item If $\beta=0$ and $k_{t\et}\left(\ollb\right)=0$, then applying (iii) of definition \ref{defi:B} to $\bb+A$ we obtain 
\begin{align*}
 L^\gamma(\olb)-L^\gamma(\bb) &=c\left(\olb\right)-c(\bb)+\gamma\esp{v(\bb+\ell_2-\ell_1-\ell_3+A)-v(\bb+A)}\\
                      &\le c\left(\ollb\right)-c\left(\olb\right) +\gamma \esp{v(\bb+2\ell_2-\ell_1-\ell_3+A)-v(\bb+\ell_2-\ell_1-\ell_3+A)}\\
                      &=L^\gamma\left(\ollb\right)-L^\gamma\left(\olb\right).
\end{align*}
\item If $\beta=0$ and $k_{t\et}\left(\ollb\right)=1$ then, using successively the fact that $v$ is an element of $\Ilq$ and of $\Ip$ we get  
\begin{align*}
 L^\gamma(\olb)-L^\gamma(\bb) &=c\left(\olb\right)-c(\bb)+\gamma\esp{v(\bb+\ell_2-\ell_1-\ell_3+A)-v(\bb+A)}\\
                      &\le c\left(\olb\right)-c(\bb)+\gamma\esp{v(\bb+\ell_2-\ell_1+A)-v(\bb+A)}\\
                      &\le c\left(\olb\right)-c(\bb)\\
                      &\le c\left(\ollb\right)-c\left(\olb\right)
                      %&\le c\left(\ollb\right)-c\left(\olb\right) +\gamma \esp{v(\bb+\ell_2-\ell_1-\ell_3+A)-v(\bb+\ell_2-\ell_1-\ell_3+A)}\\
                      =L^\gamma\left(\ollb\right)-L^\gamma\left(\olb\right).
\end{align*}
\end{itemize}
\end{itemize}

\paragraph{Assertion (ii).} By symmetry between $\ell_1$ and $\ell_3$, the proof is similar to that of assertion (i), by using Assertion (v) of definition \ref{defi:C'}.

\paragraph{Assertion (iii).} 
Set $\bb=(1,0,0,1)+\gre_i$, for $i\in\llbracket 0,3 \rrbracket$, and denote $\olb=\bb+\ell_2-\ell_1-\ell_3$ and $\ollb=\bb+2\ell_2-\ell_1-\ell_3$. 
We show that $L^\gamma(\olb)-L^\gamma(\bb)\le L^\gamma\left(\ollb\right)-L^\gamma\left(\olb\right)$ for each $i\in\llbracket 0,3 \rrbracket$.   

\begin{itemize}
\item[(iiia)] If $i=0$, then $\bb=(2,0,0,1)$, $\olb=(1,0,0,0)$ and $\ollb=(1,1,1,0)$. Then, from Assertion (ii) of definition \ref{defi:C'} we obtain that 
        \begin{align*}
        L^\gamma(\olb)-L^\gamma(\bb) &=c\left(\olb\right)-c(\bb)+\gamma\esp{v(\bb+\ell_2-\ell_1-\ell_3+A)-v(\bb+A)}\\
                                     &\le c\left(\ollb\right)-c\left(\olb\right) +\gamma \esp{v(\bb+2(\ell_2-\ell_1)-\ell_3+A)-v(\bb+\ell_2-\ell_1-\ell_3+A)}\\
                                     %&= c\left(\ollb\right)-c\left(\olb\right) +\gamma \esp{v\left(\ollb+A\right)-v\left(\olb+A\right)}\\
                                     &=L^\gamma\left(\ollb\right)-L^\gamma\left(\olb\right).
        \end{align*}   
\item[(iiib)] If $i=1$, then $\bb=(1,1,0,1)$, $\olb=(0,1,0,0)$ and $\ollb=(0,2,1,0)$. We distinguish two cases: 
    \begin{itemize}
        \item If $k_{t_{i\left(\ollb\right)}}\left(\ollb\right)=0$, then, applying assertion (ii) of Lemma \ref{lemma:B} for $\beta=0$ to $\bb-\ell_1+A$, we get that 
        \begin{align*}
        L^\gamma(\olb)-L^\gamma(\bb) &=c\left(\olb\right)-c(\bb)+\gamma\esp{v(\bb+\ell_2-\ell_1-\ell_3+A)-v(\bb-\ell_1+A)}\\
                                     &\le c\left(\ollb\right)-c\left(\olb\right) +\gamma \esp{v(\bb+2\ell_2-\ell_1-\ell_3+A)-v(\bb+\ell_2-\ell_1-\ell_3+A)}\\
                                     %&= c\left(\ollb\right)-c\left(\olb\right) +\gamma \esp{v\left(\ollb+A\right)-v\left(\olb+A\right)}\\
                                     &=L^\gamma\left(\ollb\right)-L^\gamma\left(\olb\right).
        \end{align*}   
        \item If $k_{t_{i\left(\ollb\right)}}\left(\ollb\right)=1$, then, as $v\in\Ip$ we obtain that 
        \begin{align*}
        L^\gamma(\olb)-L^\gamma(\bb) &=c\left(\olb\right)-c(\bb)+\gamma\esp{v(\bb+\ell_2-\ell_1-\ell_3+A)-v(\bb-\ell_1+A)}\\
                                     &\le c\left(\olb\right)-c(\bb)\\
                                     &\le c\left(\ollb\right)-c\left(\olb\right)\\
                                     &= c\left(\ollb\right)-c\left(\olb\right) +\gamma \esp{v(\bb+\ell_2-\ell_1-\ell_3+A)-v(\bb+\ell_2-\ell_1-\ell_3+A)}\\
                                     &=L^\gamma\left(\ollb\right)-L^\gamma\left(\olb\right).
        \end{align*}  
    \end{itemize}

\item[(iiic)] If $i=2$, then $\bb=(1,0,1,1)$, $\olb=(0,0,1,0)$ and $\ollb=(0,1,2,0)$. This case is symmetric to (iiib), using now Assertion (i) of Lemma \ref{lemma:B} for $\beta=0$ to $\bb-\ell_3+A$ if $k_{t_{i\left(\ollb\right)}}\left(\ollb\right)=0$. 

\item[(iiid)] If $i=3$, then $\bb=(1,0,0,2)$, $\olb=(0,0,0,1)$ and $\ollb=(0,1,1,1)$ and the argument is symmetric to that of (iiia), applying now Assertion (iii) 
of definition \ref{defi:C'}. 
\end{itemize}

\paragraph{\underline{Step III: $L^\gamma v \in \Cp$}.}

\paragraph{Assertion (i).} We first show that $L^\gamma$ satisfies Assertion (i) of definition \ref{defi:C'}. For this, let $\cc=(\alpha,0,\beta,0)+\gre_i$ for $\alpha\ge 2$, $\beta\ge 0$ and $i\in\llbracket 0,3 \rrbracket$. 
Let us also set $\olc=\cc+\ell_2-\ell_1$ and $\ollc=\cc+2(\ell_2-\ell_1).$  

\begin{itemize}
    \item[(ia)] If $i=0$, then we get $\cc=(\alpha+1,0,\beta,0)$, $\olc=(\alpha,0,\beta+1,0)$ and 
    $\ollc=(\alpha-1,0,\beta+2,0).$ So as $c\in\Cp$, applying Assertion (i) of definition \ref{defi:C'} to $v$ we get that
    \begin{align*}
        L^\gamma(\olc)-L^\gamma(\cc) &=c\left(\olc\right)-c(\cc)+\gamma\esp{v(\cc+\ell_2-\ell_1+A)-v(\cc+A)}\\
                                     &\le c\left(\ollc\right)-c\left(\olc\right) +\gamma \esp{v(\cc+2(\ell_2-\ell_1)+A)-v(\cc+\ell_2-\ell_1+A)}\\
                                     &=L^\gamma\left(\ollc\right)-L^\gamma\left(\olc\right).
        \end{align*} 
    \item[(ib)] If $i=1$, then $\cc=(\alpha,1,\beta,0)$, $\olc=(\alpha-1,1,\beta+1,0)$ and $\ollc=(\alpha-2,1,\beta+2,0)$ and thus 
    \begin{equation}
        \label{eq:circus1}
        L^\gamma(\olc)-L^\gamma(\cc) =c\left(\olc\right)-c(\cc)+\gamma\esp{v(\cc+\ell_2-2\ell_1+A)-v(\cc-\ell_1+A)}.
    \end{equation}
    We are in the following alternative: 
    \begin{itemize}
        \item If $\alpha=2$ and $k_{t_{i\left(\ollc\right)}}\left(\ollc\right)=0$, then applying Assertion (i) of Lemma \ref{lemma:B} to $\bb:=\cc-\ell_1+A$ we obtain from \eqref{eq:circus1} that 
        \begin{align*}
        L^\gamma(\olc)-L^\gamma(\cc) %&=c\left(\olc\right)-c(\cc)+\gamma\esp{v(\cc+\ell_2-2\ell_1+A)-v(\cc-\ell_1+A)}\\
                                     &\le c\left(\ollc\right)-c\left(\olc\right) +\gamma \esp{v(\cc+2\ell_2-2\ell_1+A)-v(\cc+\ell_2-2\ell_1+A)}\\
                                     &=L^\gamma\left(\ollc\right)-L^\gamma\left(\olc\right).
        \end{align*} 
        \item If $\alpha=2$ and $k_{t_{i\left(\ollc\right)}}\left(\ollc\right)=1$, then as $v\in \Ip$, \eqref{eq:circus1} implies 
        \begin{align*}
        L^\gamma(\olc)-L^\gamma(\cc) %&=c\left(\olc\right)-c(\cc)+\gamma\esp{v(\cc+\ell_2-2\ell_1+A)-v(\cc-\ell_1+A)}\\
                                     &\le c\left(\olc\right)-c(\cc)+\gamma\esp{v(\cc-\ell_1+A)-v(\cc-\ell_1+A)}\\
                                     &\le c\left(\ollc\right)-c\left(\olc\right)\\
                                     &= c\left(\ollc\right)-c\left(\olc\right) +\gamma \esp{v(\cc+\ell_2-2\ell_1+A)-v(\cc+\ell_2-2\ell_1+A)}\\
                                     &=L^\gamma\left(\ollc\right)-L^\gamma\left(\olc\right).
        \end{align*} 
        \item If $\alpha>2$ then, applying (i) of definition \ref{defi:C'} to $\cc-\ell_1+A$, \eqref{eq:circus1} entails that 
    \end{itemize}
    \begin{align*}
        L^\gamma(\olc)-L^\gamma(\cc) %&=c\left(\olc\right)-c(\cc)+\gamma\esp{v(\cc+\ell_2-2\ell_1+A)-v(\cc-\ell_1A)}\\
                                     &\le c\left(\ollc\right)-c\left(\olc\right) +\gamma \esp{v(\cc+2\ell_2-3\ell_1+A)-v(\cc+\ell_2-2\ell_1+A)}\\
                                     &=L^\gamma\left(\ollc\right)-L^\gamma\left(\olc\right).
        \end{align*} 
    \item[(ic)] If $i=2$, then we get $\cc=(\alpha,0,\beta+1,0)$, $\olc=(\alpha-1,0,\beta+2,0)$ and 
    $\ollc=(\alpha-2,0,\beta+3,0).$ So, as in case (ia), applying Assertion (i) of definition \ref{defi:C'} to $v$ we get that
    \begin{align*}
        L^\gamma(\olc)-L^\gamma(\cc) &=c\left(\olc\right)-c(\cc)+\gamma\esp{v(\cc+\ell_2-\ell_1+A)-v(\cc+A)}\\
                                     &\le c\left(\ollc\right)-c\left(\olc\right) +\gamma \esp{v(\cc+2(\ell_2-\ell_1)+A)-v(\cc+\ell_2-\ell_1+A)}\\
                                     &=L^\gamma\left(\ollc\right)-L^\gamma\left(\olc\right).
        \end{align*} 
    \item[(id)] If $i=3$, then $\cc=(\alpha,0,\beta,1)$, $\olc=(\alpha-1,0,\beta+1,1)$ and $\ollc=(\alpha-2,0,\beta+2,1)$. We have the following cases: 
    \begin{itemize}
        \item If $\beta>0$, then applying Assertion (i) of definition \ref{defi:C'} to $\cc-\ell_3$ we get that 
        \begin{align*}
        L^\gamma(\olc)-L^\gamma(\cc) &=c\left(\olc\right)-c(\cc)+\gamma\esp{v(\cc+\ell_2-\ell_1-\ell_3+A)-v(\cc-\ell_3+A)}\\
                                     &\le c\left(\ollc\right)-c\left(\olc\right) +\gamma \esp{v(\cc+2\ell_2-2\ell_1-\ell_3+A)-v(\cc+\ell_2-\ell_1-\ell_3+A)}\\
                                     &=L^\gamma\left(\ollc\right)-L^\gamma\left(\olc\right).
        \end{align*} 
        \item If $\beta=0$, then Assertion (ii) of definition \ref{defi:C'} applied to $\cc$, shows that
        \begin{align*}
        L^\gamma(\olc)-L^\gamma(\cc) &=c\left(\olc\right)-c(\cc)+\gamma\esp{v(\cc+\ell_2-\ell_1-\ell_3+A)-v(\cc+A)}\\
                                     &\le c\left(\ollc\right)-c\left(\olc\right) +\gamma \esp{v(\cc+2\ell_2-2\ell_1-\ell_3+A)-v(\cc+\ell_2-\ell_1-\ell_3+A)}\\
                                     &=L^\gamma\left(\ollc\right)-L^\gamma\left(\olc\right).
        \end{align*}    
    \end{itemize}
\end{itemize}
This shows that $L^\gamma v$ satisfies Assertion (i) of definition \ref{defi:C'}.

\paragraph{Assertion (ii).} We now show that $L^\gamma$ satisfies Assertion (ii) of definition \ref{defi:C'}. 
Let $\cc=(\alpha,0,0,1)+\gre_i$ for $\alpha\ge 2$ and $i\in\llbracket 0,3 \rrbracket$. 
Let us also now set $\olc=\cc+\ell_2-\ell_1-\ell_3$ and $\ollc=\cc+2(\ell_2-\ell_1)-\ell_3.$ 
\begin{itemize}
    \item[(iia)] If $i=0$, then $\cc=(\alpha+1,0,0,1)$, $\olc=(\alpha,0,0,0)$ and 
    $\ollc=(\alpha-1,0,1,0).$ So, as $v$ satisfies (ii) of definition \ref{defi:C'}, 
    \begin{align*}
        L^\gamma(\olc)-L^\gamma(\cc) &=c\left(\olc\right)-c(\cc)+\gamma\esp{v(\cc+\ell_2-\ell_1-\ell_3+A)-v(\cc+A)}\\
                                     &\le c\left(\ollc\right)-c\left(\olc\right) +\gamma \esp{v(\cc+2(\ell_2-\ell_1)-\ell_3+A)-v(\cc+\ell_2-\ell_1-\ell_3+A)}\\
                                     &=L^\gamma\left(\ollc\right)-L^\gamma\left(\olc\right).
        \end{align*} 
    \item[(iib)] If $i=1$ we obtain $\cc=(\alpha,1,0,1)$, $\olc=(\alpha-1,1,0,0)$ and $\ollc=(\alpha-2,1,1,0).$ Then, 
    \begin{equation}
        \label{eq:circus2}
        L^\gamma(\olc)-L^\gamma(\cc) =c\left(\olc\right)-c(\cc)+\gamma\esp{v(\cc+\ell_2-2\ell_1-\ell_3+A)-v(\cc-\ell_1+A)}.
    \end{equation}
    Then we are in the following alternative: 
    \begin{itemize}
        \item If $\alpha=2$ and $k_{t_{i\left(\ollc\right)}}\left(\ollc\right)=0$ then, applying (iii) of definition \ref{defi:B} to 
        $\bb=\cc-\ell_1+A$, it follows from \eqref{eq:circus2} that 
        \begin{align*}
        L^\gamma(\olc)-L^\gamma(\cc) %&=c\left(\olc\right)-c(\cc)+\gamma\esp{v(\cc+\ell_2-2\ell_1-\ell_3+A)-v(\cc-\ell_1+A)}\\
                                     &\le c\left(\ollc\right)-c\left(\olc\right) +\gamma \esp{v(\cc+2\ell_2-2\ell_1-\ell_3+A)-v(\cc+\ell_2-2\ell_1-\ell_3+A)}\\
                                     &=L^\gamma\left(\ollc\right)-L^\gamma\left(\olc\right).
        \end{align*} 
        \item If $\alpha=2$ and $k_{t_{i\left(\ollc\right)}}\left(\ollc\right)=1$, then as $v\in\Ilq\cap\Ip$, \eqref{eq:circus2} 
        implies 
        \begin{align*}
        L^\gamma(\olc)-L^\gamma(\cc) %&=c\left(\olc\right)-c(\cc)+\gamma\esp{v(\cc+\ell_2-2\ell_1+A)-v(\cc-\ell_1+A)}\\
                                     &\le c\left(\olc\right)-c(\cc)+\gamma\esp{v(\cc-\ell_1+\ell_2-\ell_1+A)-v(\cc-\ell_1+A)}\\
                                     &\le c\left(\olc\right)-c(\cc)\\
                                     %&\le c\left(\ollc\right)-c\left(\olc\right)\\
                                     &= c\left(\ollc\right)-c\left(\olc\right) +\gamma \esp{v(\cc+\ell_2-2\ell_1-\ell_3+A)-v(\cc+\ell_2-2\ell_1-\ell_3+A)}\\
                                     &=L^\gamma\left(\ollc\right)-L^\gamma\left(\olc\right).
        \end{align*} 
        \item If $\alpha>2$, we apply (ii) of definition \ref{defi:C'} to $\cc-\ell_1+A$. Then \eqref{eq:circus2} implies that 
        \begin{align*}
        L^\gamma(\olc)-L^\gamma(\cc) &=c\left(\olc\right)-c(\cc)+\gamma\esp{v(\cc+\ell_2-2\ell_1-\ell_3+A)-v(\cc-\ell_1+A)}\\
                                     &\le c\left(\ollc\right)-c\left(\olc\right) +\gamma \esp{v(\cc+2\ell_2-3\ell_1-\ell_3+A)-v(\cc+\ell_2-2\ell_1-\ell_3+A)}\\
                                     &=L^\gamma\left(\ollc\right)-L^\gamma\left(\olc\right).
        \end{align*} 
    \end{itemize}
    \item[(iic)] If $i=2$, then we get $\cc=(\alpha,0,1,1)$, $\olc=(\alpha-1,0,1,0)$ and 
    $\ollc=(\alpha-2,0,2,0).$ So, applying Assertion (i) of definition \ref{defi:C'} for $\cc-\ell_3+A$ yields to 
    \begin{align*}
        L^\gamma(\olc)-L^\gamma(\cc) &=c\left(\olc\right)-c(\cc)+\gamma\esp{v(\cc+\ell_2-\ell_1-\ell_3+A)-v(\cc-\ell_3+A)}\\
                                     &\le c\left(\ollc\right)-c\left(\olc\right) +\gamma \esp{v(\cc+2(\ell_2-\ell_1)-\ell_3+A)-v(\cc+\ell_2-\ell_1-\ell_3+A)}\\
                                     &=L^\gamma\left(\ollc\right)-L^\gamma\left(\olc\right).
        \end{align*} 
    \item[(iid)] If $i=3$, then $\cc=(\alpha,0,0,2)$, $\olc=(\alpha-1,0,0,1)$ and $\ollc=(\alpha-2,0,1,1)$. So applying (v) of definition 
    \ref{defi:C'} we get that 
        \begin{align*}
        L^\gamma(\olc)-L^\gamma(\cc) &=c\left(\olc\right)-c(\cc)+\gamma\esp{v(\cc+\ell_2-\ell_1-\ell_3+A)-v(\cc+A)}\\
                                     &\le c\left(\ollc\right)-c\left(\olc\right) +\gamma \esp{v(\cc+2\ell_2-2\ell_1-2\ell_3+A)-v(\cc+\ell_2-\ell_1-\ell_3+A)}\\
                                     &=L^\gamma\left(\ollc\right)-L^\gamma\left(\olc\right).
        \end{align*} 
\end{itemize}

\paragraph{Assertions (iii) and (iv).} By the symmetry between $\ell_1$ and $\ell_3$, the proofs are analog to that of assertions (i) and (ii), respectively. 

\paragraph{Assertion (v).} We conclude by proving that $L^\gamma v$ satisfies assertion (v) of definition \ref{defi:C'}. 
We let $\cc=(\alpha,0,0,\beta)+\gre_i$ for  $i\in \llbracket 0,3 \rrbracket$, $\alpha\ge 2$ and $\beta\ge 2$, and set 
$\olc=\cc+\ell_2-\ell_1-\ell_3$ and $\ollc=\cc+2(\ell_2-\ell_1-\ell_3)$. 
\begin{itemize}
    \item[(va)] If $i=0$, then $\cc=(\alpha+1,0,0,\beta)$, $\olc=(\alpha,0,0,\beta-1)$ and $\ollc=(\alpha-1,0,0,\beta-2)$. Then, as 
    $v$ satisfies assertion (v) of definition \ref{defi:C'} we obtain that 
    \begin{align*}
        L^\gamma(\olc)-L^\gamma(\cc) &=c\left(\olc\right)-c(\cc)+\gamma\esp{v(\cc+\ell_2-\ell_1-\ell_3+A)-v(\cc+A)}\\
                                     &\le c\left(\ollc\right)-c\left(\olc\right) +\gamma \esp{v(\cc+2\ell_2-2\ell_1-2\ell_3+A)-v(\cc+\ell_2-\ell_1-\ell_3+A)}\\
                                     &=L^\gamma\left(\ollc\right)-L^\gamma\left(\olc\right).
        \end{align*} 
    \item[(vb)] If $i=1$, then $\cc=(\alpha,1,0,\beta)$, $\olc=(\alpha-1,1,0,\beta-1)$ and $\ollc=(\alpha-2,1,0,\beta-2)$. 
    Then we are in the following alternative: 
    \begin{itemize}
        \item If $\alpha=2$, then, applying assertion (iv) of definition \ref{defi:C'} to $v$ and $\cc-\ell_1+A$ leads to 
        \begin{align*}
        L^\gamma(\olc)-L^\gamma(\cc) &=c\left(\olc\right)-c(\cc)+\gamma\esp{v(\cc+\ell_2-2\ell_1-\ell_3+A)-v(\cc-\ell_1+A)}\\
                                     &\le c\left(\ollc\right)-c\left(\olc\right) +\gamma \esp{v(\cc+2\ell_2-2\ell_1-2\ell_3+A)-v(\cc+\ell_2-2\ell_1-\ell_3+A)}\\
                                     &=L^\gamma\left(\ollc\right)-L^\gamma\left(\olc\right).
        \end{align*}
        \item If $\alpha>2$, then we apply assertion (v) of definition \ref{defi:C'} to $v$ and $\cc-\ell_1+A$, and get that 
        \begin{align*}
        L^\gamma(\olc)-L^\gamma(\cc) &=c\left(\olc\right)-c(\cc)+\gamma\esp{v(\cc+\ell_2-2\ell_1-\ell_3+A)-v(\cc-\ell_1+A)}\\
                                     &\le c\left(\ollc\right)-c\left(\olc\right) +\gamma \esp{v(\cc+2\ell_2-3\ell_1-2\ell_3+A)-v(\cc+\ell_2-2\ell_1-\ell_3+A)}\\
                                     &=L^\gamma\left(\ollc\right)-L^\gamma\left(\olc\right).
        \end{align*}
    \end{itemize}
    \item[(vc)] The case $i=2$ is symmetric to (vb). 
    \item[(vd)] The case $i=3$ is symmetric to (va). 
\end{itemize}
This completes the proof that $L^\gamma v \in \Cp$. 
\end{proof}

\subsection{Concluding the proof of Theorem \ref{thm:main}}
\label{subsec:proof}
We check that the assumptions of Theorem \ref{theocentral} are satisfied, for the set $\mathscr V$ which has been shown to be stable under the action of $L^\gamma$ in Proposition \ref{prop:stable}. It is also immediate that it is stable under point-wise convergence, which proves assertion 1). Setting $\mathscr D\equiv \mathscr T$, assertion 2) is granted by Proposition \ref{propseuilopti}. 

We now check that the cost function $c$ satisfies assumptions (i)-(iii) of Theorem \ref{theocentral}. For this, let 
\[w=\fonccc{\N^4}{\R_+}{x}{x_0+x_1+x_2+x_3+1}.\] 

\begin{itemize}
    \item[(i)] For all $\x,\bu\in\N^4$ such that $\x-\bu\in\N^4$, we have $c(\x)/w(\x)\leq \max(c_0,c_1,c_2,c_3)$. Moreover, there is a single arrival at each time step, implying that 
    \begin{align*}
    \dfrac{1}{w(\x)}\dsum_{\by} \PP(\by | \x-\bu)w(\by) = \esp{\dfrac{w(\x-\bu+\ba)}{w(\x)} \Big| A=\ba}  
    &\leq \esp{\dfrac{w(\x+\ba)}{w(\x)}\Big| A=\ba}\\ 
    &= \dfrac{\dsum_{i=0}^3x_i+2 }{\dsum_{i=0}^3x_i + 1} \leq 2. 
    \end{align*}
    \item[(ii)] Let $\pi_0$ be the  ``no matching'' policy. With the same argument, we show that for all $J$-steps policy $\pi$, for all $\x$, 
    $$\mathbb E^\pi_\x\left[w(X_p)\right] \le \mathbb E^{\pi_0}_\x\left[w(X_p)\right] \le w(\x)+p,$$
    %$$\dsum_{y}\PP_\pi (y | x) w(y)\leq \dsum_{y}\PP_{\pi_0} (y | x) w(y)\leq w(x)+J$$
and so 
\[\dfrac{\varepsilon^p}{w(\x)}\mathbb E^\pi_\x\left[w(X_p)\right]\leq \dfrac{\varepsilon^p(w(\x)+p)}{w(\x)}\leq \varepsilon^p(p+1),\]
and it is sufficient to choose $p\in\N$ such that $\varepsilon^p(p+1)<\dfrac{1}{2}<1$ to conclude.  
\item[(iii)] The last condition is satisfied because in the present case, from Lemma \ref{lemme35} for all $v$ and $\x$, 
$L^\gamma v(\x)$ by taking the lower bound over a finite subset. 
\end{itemize}
This completes the proof. 

\nocite{*}
\bibliographystyle{plain}
\bibliography{graphe_en_N}

\end{document}